\documentclass[reqno]{amsart}
\usepackage{amssymb, amsmath, amsthm, amscd, mathrsfs}

\usepackage{paralist}
 
\usepackage{bigints}
\usepackage{dsfont}
\usepackage{empheq}
\usepackage{amssymb}
\usepackage{cases}
\usepackage{enumitem}
\usepackage{cite}
\allowdisplaybreaks

\usepackage[colorlinks]{hyperref}
\usepackage[colorinlistoftodos]{todonotes}

\theoremstyle{plain}
\newtheorem{theorem}{Theorem}[section]

\newtheorem{lemma}[theorem]{Lemma}
\newtheorem{proposition}{Proposition}

\theoremstyle{definition}

\newtheorem{remark}{Remark}

\def \dv {\mathrm{div}}
\def \d {\mathrm{d}}

\title[Stable determination of coefficients in semilinear parabolic system] 
      {Stable determination of coefficients in semilinear parabolic system with dynamic boundary conditions}

\author{E. M. Ait Ben Hassi}
\author{S. E. Chorfi}
\author{L. Maniar}

\address{E. M. Ait Ben Hassi, S. E. Chorfi and L. Maniar, Cadi Ayyad University, Faculty of Sciences Semlalia, LMDP, UMMISCO (IRD-UPMC), B.P. 2390, Marrakesh, Morocco}

\email{m.benhassi@gmail.com, chorphi@gmail.com, maniar@uca.ma}


\begin{document}
\begin{abstract}
In this work, we study the stable determination of four space-dependent coefficients appearing in a coupled semilinear parabolic system with variable diffusion matrices subject to dynamic boundary conditions which couple intern-boundary phenomena. We prove a Lipschitz stability result for interior and boundary potentials by means of only one observation component, localized in any arbitrary open subset of the physical domain. The proof mainly relies on some new Carleman estimates for dynamic boundary conditions of surface diffusion type.
\end{abstract}

\maketitle

\section{Introduction and main results}

In the present paper, we deal with an inverse problem of determining interior and boundary coupling coefficients in a semilinear parabolic system with dynamic boundary conditions. We consider $\Omega \subset \mathbb{R}^N$, $N= 2, 3$, a nonempty bounded domain with boundary $\Gamma=\partial \Omega$ of class $C^2$. We denote by $\nu$ the normal vector field on $\Gamma$ pointing outwards the domain $\Omega$. Let $T>0$ and let us set
$$\Omega_T =(0,T)\times \Omega \qquad \text{ and } \qquad \Gamma_T=(0,T)\times \Gamma.$$
We consider the following semilinear system of coupled parabolic equations with dynamic boundary conditions
\begin{empheq}[left =\empheqlbrace]{alignat=2}
\begin{aligned}
&\partial_t y = \dv(A_1(x) \nabla y) + p_{11}(x)y + p_{12}(x)z + p_{13}(x)f(y,z), &\text{in } \Omega_T , \\
&\partial_t z = \dv(A_2(x) \nabla z) + p_{21}(x)y + p_{22}(x)z, &\text{in } \Omega_T , \\
&\partial_t y_{\Gamma}= \dv_{\Gamma} (D_1(x)\nabla_\Gamma y_{\Gamma}) -\partial_{\nu}^{A_1} y + q_{11}(x)y_{\Gamma} + q_{12}(x)z_{\Gamma}\\
& \hspace{1cm}+ q_{13}(x) g(y_{\Gamma},z_{\Gamma}), &\text{on } \Gamma_T, \\
&\partial_t z_{\Gamma}= \dv_{\Gamma} (D_2(x)\nabla_\Gamma z_{\Gamma}) -\partial_{\nu}^{A_2} z + q_{21}(x)y_{\Gamma} + q_{22}(x)z_{\Gamma}, &\text{on } \Gamma_T, \\
& y_{\Gamma}(t,x) = y_{|\Gamma}(t,x), \qquad  z_{\Gamma}(t,x) = z_{|\Gamma}(t,x),  &\text{on } \Gamma_T, \\
&(y,y_{\Gamma})\rvert_{t=0}=(y_0, y_{0,\Gamma}), \qquad  (z,z_{\Gamma})\rvert_{t=0}=(z_0, z_{0,\Gamma}),   &\Omega\times\Gamma,  \label{eq1to6}
\end{aligned}
\end{empheq}
where $(y_0, y_{0,\Gamma}), (z_0, z_{0,\Gamma})\in L^2(\Omega)\times L^2(\Gamma)$ are the initial states. We emphasize that $(y_{0,\Gamma}, z_{0,\Gamma})$ is not necessarily the trace on $\Gamma$ of $(y_0,z_0)$ (we do not assume that the latter has a trace), but if the trace of $(y_0,z_0)$ is well-defined, then the trace must coincide with $(y_{0,\Gamma}, z_{0,\Gamma})$. All potentials are assumed to be bounded, that is, $p_{ij} \in L^\infty(\Omega)$ and $q_{ij} \in L^\infty(\Gamma)$. The nonlinearities $f,g \colon \mathbb{R}^2 \rightarrow \mathbb{R}$ are assumed to be Lipschitz continuous with respect to the two variables. We assume that the diffusion matrices $A_k$ and $D_k$, $k=1,2,$ are symmetric and uniformly elliptic, i.e.,
\begin{align}
A_k=(a^k_{ij})_{i,j} \in C^1\left(\overline{\Omega}; \mathbb{R}^{N\times N}\right), \quad a^k_{ij}=a^k_{ji}, \quad 1\leq i,j\leq N, \label{symA}\\
D_k=(d^k_{ij})_{i,j} \in C^1\left(\Gamma; \mathbb{R}^{N\times N}\right), \quad d^k_{ij}=d^k_{ji}, \quad 1\leq i,j\leq N, \label{symD}
\end{align}
and there exist constants $\beta, \beta_\Gamma >0$ such that
\begin{align}
A_k(x)\zeta \cdot \zeta &\geq \beta |\zeta|^2, && x\in \overline{\Omega}, \;\zeta \in \mathbb{R}^N, \label{uellipA}\\
\langle D_k(x)\zeta, \zeta \rangle_{\Gamma} &\geq \beta_\Gamma |\zeta|_\Gamma^2, && x\in \Gamma, \;\zeta \in \mathbb{R}^N, \label{uellipD}
\end{align}
where $``\cdot"$ denotes the Euclidean inner product on $\mathbb{R}^N$ while $\langle \cdot, \cdot\rangle_\Gamma$ is the Riemannian inner product on $\Gamma$. By $y_{|\Gamma}$, one designates the trace of $y$, and by $\partial_{\nu}^{A_k} y :=(A_k\nabla y\cdot \nu)_{|\Gamma}$ the conormal derivative. The operator $\dv=\dv_x$ stands for the Euclidean divergence operator in $\Omega$. If $g$ denotes the natural Riemannian metric on $\Gamma$, $\dv_\Gamma$ stands for the divergence operator in $(\Gamma, g)$. The term $\dv_\Gamma\left(D_k(\cdot)\nabla y_\Gamma\right)$ represents the surface diffusion given by the vector field $D_k(\cdot)\nabla y_\Gamma: \Gamma \rightarrow T\Gamma,$ where $T\Gamma =\bigcup\limits_{x\in \Gamma} T_x\Gamma$ and $T_x\Gamma$ is the tangent space of $\Gamma$ at $x$. For any smooth function $y \colon \overline{\Omega} \rightarrow \mathbb{R}$, the tangential gradient $\nabla_\Gamma y$ is the part of the Euclidean gradient $\nabla y$ tangent to $\Gamma$, namely,
$$\nabla_\Gamma y = \nabla y -(\partial_\nu y) \nu.$$
Since $\Gamma$ is a compact Riemannian manifold without boundary, the following divergence formula holds
\begin{equation}\label{sdt}
\int_\Gamma (\dv_\Gamma X)z \,\d S =- \int_\Gamma \langle X, \nabla_\Gamma z \rangle_\Gamma \,\d S, \qquad z\in H^1(\Gamma),
\end{equation}
where $X$ is any $C^1$ vector field and $\d S$ is the surface measure on $\Gamma$.

\subsection{Physical motivation}
Dynamic boundary conditions (also known as generalized Wentzell boundary conditions) contain the time derivative on the boundary, which often simulates an imperfect contact with a given interface. In contrast to static boundary conditions (e.g., Dirichlet, Neumann or Robin conditions), dynamic boundary conditions arise naturally as part of the physical derivation when incorporating boundary conditions into the formulation of the problem. We refer to \cite{Go'06} for the derivation of this type of boundary conditions. A more recent and rather different approach which uses the Carslaw-Jaeger law can be found in \cite{Sa'20}. Parabolic systems with dynamic boundary conditions have received considerable attention in the last years \cite{BCMO'20, FGGR'02, KM'19, KMMR'19, KMO'22, MMS'17}. They appear in various mathematical models in population dynamics \cite{FH'11}, in heat transfer and diffusion phenomena, especially when the diffusion lies in the interface between a solid and a moving fluid, see \cite{La'32} and the references therein.

Semilinear coupled bulk-surface systems such as \eqref{eq1to6} arise in biological models of cells, chemistry, solid/fluid mechanics, and so on. For instance, the authors of \cite{EFPT'18} study a mathematical model for asymmetric stem cell division. In \cite{RR'12}, a relevant model is proposed for signaling molecules in a cell. Both papers present theoretical results on wellposedness and stability as well as numerical results on the discretization. We refer to \cite{Mie'13} and the references therein for some general models of bulk-interface interaction in thermomechanics. It is worthwhile to mention the recent paper \cite{MS'20} (see its references for more applications) which investigates the global wellposedness of semilinear systems like \eqref{eq1to6} with constant diffusion matrices. In particular, they provide some examples related to the Brusselator system as well as chemical reactions in cell division. In the previous applications, the coupling potentials $(p_{13}, q_{13}, p_{21}, q_{21})$ represent some unknown concentrations, reaction rates or radiative quantities either in the bulk $\Omega$ or on its surface $\Gamma$. These parameters play a crucial role in describing the underlying model and need to be determined from some experimental measurements.

\subsection{Literature on coefficient inverse problems}
Coefficient inverse problems for parabolic systems have been extensively studied in the literature and most results rely on suitable Carleman estimates. Yamamoto and Zou in 2001 have adapted the Bukhgeim-Klibanov method \cite{BK'81} to prove a Lipschitz stability result for the radiative potential in a single heat equation with Dirichlet boundary conditions \cite{YZ'01}. Later on, similar results were established by Choi \cite{Ch'03} for the same equation with space-periodic boundary conditions and coefficients. A similar problem for a nonlinear parabolic equation with periodic coefficients has been considered by Kaddouri and Teniou in \cite{KT'14}. In \cite{BP'02}, Baudouin and Puel have established a stability result for a spatial potential appearing in the Schrödinger equation. An inverse problem for the same equation with discontinuous coefficients has been studied by Baudouin and Mercado in \cite{BM'08}. Further stability results were obtained in \cite{MOR'08} for the potential from some less restrictive boundary and interior measurements. The authors developed some Carleman estimates under a relaxed pseudoconvexity condition allowing degenerate weights. Uniqueness and stability results were obtained by Benabdallah et al. in \cite{BGL'07} for a discontinuous diffusion coefficient and an initial datum in a heat equation by using a boundary measurement.

As for coupled systems, Cristofol et al. have proven in \cite{CGR'06} some stability results for a coefficient and for initial data by using some internal observed data. Their proof is based on a Carleman estimate with one observation component. An improvement of the previous work has been established for a nonlinear parabolic system in \cite{CGRY'12}. In \cite{BCGY'09}, some stability estimates were obtained by Benabdallah et al. for parabolic systems with first and zeroth order coupling coefficients. They proved Lipschitz stability estimates for some or all of the coefficients by using one observation component and without data of initial conditions.

In spite of the large literature for PDEs with static boundary conditions \cite{Is'17, KL'13}, inverse problems for parabolic systems with dynamic boundary conditions have not been much studied even for linear systems. Recently in \cite{ACMO'20, Ch'21}, the authors have considered an inverse source problem for a linear parabolic system with dynamic boundary conditions and established a Lipschitz stability estimate for the forcing terms from internal measurements. A numerical approach has been investigated in \cite{ACM'21'} for reconstructing the source term from the final overdetermination. In \cite{ACM'21}, we have studied an inverse problem of radiative potentials and initial temperatures in a linear parabolic system with dynamic boundary conditions. We have proven a Lipschitz stability estimate for internal and boundary potentials. Regarding the semilinear case of coupled systems, Sakthivel et al. \cite{SAB'20} have recently proven a Lipschitz stability result for a coefficient in a semilinear tumor growth model from boundary measurements. We refer to the paper \cite{CGRY'12} for Dirichlet conditions, which develops a strategy based on a Carleman estimate obtained in \cite{Fu'00} along with a positivity result from \cite{Sm'83} given by the method of invariant regions. In our case, the situation is quite complicated due to the coupling of intern-boundary phenomena in addition to the coupling by the coefficients. Consequently, our results draw on new developed Carleman estimates in which we had to absorb several boundary terms. Moreover, we prove the positivity result in a different and rather simple way.

\subsection{Statement of the main result}
We are in position to formulate the problem and state the main result. We denote by $\mathbb{L}^\infty=L^\infty(\Omega)\times L^\infty(\Gamma)$, and for a fixed constant $R>0$, we consider the set of admissible potentials
\begin{align}
\mathcal{P}:=\{(p,q) \in \mathbb{L}^\infty \colon \|p\|_\infty ,\|q\|_\infty \leq R\}.
\end{align}
We are interested in the simultaneous determination of the coupling coefficients using only one observation component, namely, the identification of the potentials
$$(p_{13}, q_{13}) \qquad \text{ and } \qquad (p_{21}, q_{21})$$
in system \eqref{eq1to6} belonging to $\mathcal{P}$, from the measurement
$$z\rvert_{(t_0,t_1) \times \omega},$$
where $(t_0,t_1) \subset (0,T), \; t_0<t_1,$ and $\omega \Subset \Omega$ is a nonempty open set.

We denote the Lebesgue measure on $\Omega$ by $\d x$ and the surface measure on $\Gamma$ by $\d S$. Let us introduce the real Hilbert space
$$\mathbb{L}^2 :=L^2(\Omega, \d x)\times L^2(\Gamma, \d S)$$
equipped with the norm
$$
\left\|\left(y, y_{\Gamma}\right)\right\|_{\mathbb{L}^2}=\left\langle\left(y, y_{\Gamma}\right),\left(y, y_{\Gamma}\right)\right\rangle_{\mathbb{L}^2}^{1 / 2},$$
with the inner product
$$\langle (y,y_\Gamma),(z,z_\Gamma)\rangle_{\mathbb{L}^2} =\langle y,z\rangle_{L^2(\Omega)} +\langle y_\Gamma,z_\Gamma\rangle_{L^2(\Gamma)}.$$
Also, we consider the Hilbert space
$$\mathbb{H}^2 := \{(u,u_{\Gamma}) \in H^2(\Omega) \times H^2(\Gamma): u_{|\Gamma}=u_{\Gamma} \}$$
equipped with the norm
$$
\left\|\left(y, y_{\Gamma}\right)\right\|_{\mathbb{H}^{2}}=\left\langle\left(y, y_{\Gamma}\right),\left(y, y_{\Gamma}\right)\right\rangle_{\mathbb{H}^{2}}^{1 / 2},
$$
with the inner product
$$\left\langle\left(y, y_{\Gamma}\right),\left(z, z_{\Gamma}\right)\right\rangle_{\mathbb{H}^{2}}=\langle \Delta y, \Delta z\rangle_{L^2(\Omega)} + \langle\Delta_{\Gamma} y_{\Gamma}, \Delta_{\Gamma} z_{\Gamma}\rangle_{L^2(\Gamma)} +\langle y_\Gamma,z_\Gamma\rangle_{L^2(\Gamma)}.$$

We denote by $(y,z,y_\Gamma, z_\Gamma)$ (resp. $(\widetilde{y}, \widetilde{z},\widetilde{y}_\Gamma, \widetilde{z}_\Gamma)$) the solution of \eqref{eq1to6} corresponding to $(p_{13}, q_{13}, p_{21}, q_{21}, y_0,z_0,y_{0,\Gamma}, z_{0,\Gamma})$ (resp. $(\widetilde{p}_{13}, \widetilde{q}_{13}, \widetilde{p}_{21}, \widetilde{q}_{21}, \widetilde{y}_0,\widetilde{z}_0,\widetilde{y}_{0,\Gamma}, \widetilde{z}_{0,\Gamma})$). We shall make the following assumptions:\\
\textbf{Assumption I.}
\begin{itemize}
\item[(i)] $(p_{ij},q_{ij}), (\widetilde{p}_{13},\widetilde{q}_{13}), (\widetilde{p}_{21},\widetilde{q}_{21}) \in \mathcal{P}$, for $i=1,2$ and $j=1,2,3$.
\item[(ii)] There exist constants $r>0$ and $p_0>0$ such that
\begin{align*}
&\widetilde{y}_0, \widetilde{y}_{0,\Gamma} \ge r\quad \text{ and } \quad \widetilde{z}_0, \widetilde{z}_{0,\Gamma} \ge 0,\\
&p_{11}r + p_{12} \widetilde{z}_0 + \widetilde{p}_{13}f(r,\widetilde{z}_0) \ge 0,\\
&q_{11}r + q_{12} \widetilde{z}_{0,\Gamma} + \widetilde{q}_{13}g(r,\widetilde{z}_{0,\Gamma}) \ge 0,\\
&p_{21}, q_{21}\ge p_0 \quad \text{and} \quad\widetilde{p}_{21}, \widetilde{q}_{21}\ge p_0.
\end{align*}
\end{itemize}
Let us set $\theta=\frac{t_0+t_1}{2}$. We emphasize that Assumption I-(ii) is made to ensure that
\begin{equation}\label{pineq}
\widetilde{y}(\theta,\cdot) \ge r \text{ in } \Omega \qquad \text{and}\qquad \widetilde{y}_\Gamma(\theta,\cdot) \ge r \text{ on } \Gamma.
\end{equation}
\textbf{Assumption II.}
\begin{itemize}
\item[(i)] $f,g\in W^{1,\infty}\left(\mathbb{R}^2\right)$.
\item[(ii)] $\exists r_1>0 \colon \; |f(\widetilde{y},\widetilde{z})(\theta, \cdot)| \ge r_1\;$ and $\; |g(\widetilde{y}_\Gamma, \widetilde{z}_\Gamma)(\theta, \cdot)| \ge r_1 $.
\item[(iii)] $\partial_t f(\widetilde{y}, \widetilde{z}) \in L^2\left(t_0, t_1; H^2(\Omega)\right)\;$ and $\; \partial_t g(\widetilde{y}_\Gamma, \widetilde{z}_\Gamma) \in L^2\left(t_0, t_1; H^2(\Gamma)\right)$.
\end{itemize}
Assumption II will allow us to absorb some terms in the proof of the stability estimate. An example of semilinearities satisfying the above assumptions is given by $(y,z) \mapsto y^d z^\delta$, where $d$ and $\delta$ are nonnegative constants.

We mainly aim to establish the following Lipschitz stability estimate.
\begin{theorem}\label{thm1}
Let assumptions Assumption I and Assumption II be satisfied. Denote by $\widetilde{Y}_0:=(\widetilde{y}_0, \widetilde{y}_{0,\Gamma})$ and $\widetilde{Z}_0:=(\widetilde{z}_0, \widetilde{z}_{0,\Gamma})$. We further assume that $\widetilde{Y}_0, \widetilde{Z}_0 \in \mathbb{H}^2$ and $(y,z)(\theta, \cdot)=(\widetilde{y},\widetilde{z})(\theta, \cdot)$ in $\Omega$. Then there exists a positive constant $C=C(\Omega, \omega, p_0, \theta, t_0, t_1, r, r_1, R)$ such that
\begin{equation}\label{nstab}
\|(p_{21}-\widetilde{p}_{21},q_{21}-\widetilde{q}_{21})\|_{\mathbb{L}^2} + \|(p_{13}-\widetilde{p}_{13},q_{13}-\widetilde{q}_{13})\|_{\mathbb{L}^2}\le C \|\partial_t z - \partial_t \widetilde{z}\|_{L^2(\omega_{t_0, t_1})},
\end{equation}
where $$\omega_{t_0, t_1} :=(t_0,t_1) \times \omega.$$
\end{theorem}
Consequently, we infer the following uniqueness result.
\begin{proposition}
Under the same assumptions of Theorem \ref{thm1}, if
$$\partial_t z = \partial_t \widetilde{z} \qquad \text{ a.e in } \omega_{t_0, t_1},$$
then
\begin{align*}
p_{21} &=\widetilde{p}_{21} \qquad \text{and} \qquad p_{13} =\widetilde{p}_{13} \quad \text{a.e in} \quad  \Omega, \\
q_{21} &=\widetilde{q}_{21} \qquad \text{and} \qquad q_{13} =\widetilde{q}_{13} \quad \text{a.e on}\quad \Gamma.
\end{align*}
\end{proposition}

The paper is organized as follows. In Section \ref{sec2}, we prove a positivity result that is needed in the sequel. In Section \ref{sec3}, we prove a new Carleman estimate for the coupled parabolic system with data of only one component of the solution. In Section \ref{sec4}, we apply the obtained Carleman estimate to prove the Lipchitz stability result for the coupling coefficients using only one observation component. Finally, in Appendix \ref{app}, we sketch the proof of a Carleman estimate for a single parabolic equation with dynamic boundary conditions and a general parameter.

\section{Positivity of the solution}\label{sec2}
In this section, we prove a positivity result of the solution to a general semilinear coupled parabolic system with dynamic boundary conditions. We consider the solution of the following system.
\begin{empheq}[left =\empheqlbrace]{alignat=2}
\begin{aligned}
&\partial_t y = \dv(A_1(x) \nabla y) + f_1(y,z), &\text{in } \Omega_T , \\
&\partial_t z = \dv(A_2(x) \nabla z)  + f_2(y,z), &\text{in } \Omega_T , \\
&\partial_t y_{\Gamma}= \dv_{\Gamma} (D_1(x)\nabla_\Gamma y_{\Gamma}) -\partial_{\nu}^{A_1} y + g_1(y_\Gamma, z_\Gamma), &\text{on } \Gamma_T, \\
&\partial_t z_{\Gamma}= \dv_{\Gamma} (D_2(x)\nabla_\Gamma z_{\Gamma}) -\partial_{\nu}^{A_2} z + g_2(y_\Gamma, z_\Gamma), &\text{on } \Gamma_T, \\
& y_{\Gamma}(t,x) = y_{|\Gamma}(t,x), \qquad  z_{\Gamma}(t,x) = z_{|\Gamma}(t,x),  &\text{on } \Gamma_T, \\
&(y,y_{\Gamma})\rvert_{t=0}=(y_0, y_{0,\Gamma}), \qquad  (z,z_{\Gamma})\rvert_{t=0}=(z_0, z_{0,\Gamma}),   &\Omega\times\Gamma,  \label{peq1to6}
\end{aligned}
\end{empheq}
where $f_k, g_k \colon \mathbb{R}^2 \rightarrow \mathbb{R}, k=1,2,$ are Lipschitz continuous functions with respect to the two variables.

Recall that $\mathbb{L}^2$ is a Hilbert lattice, and its positive cone is the product of the positive cones of $L^2(\Omega)$ and $L^2(\Gamma)$. For $\mathfrak{u}:=(u,u_\Gamma)\in \mathbb{L}^2 $, we denotes by $\mathfrak{u}^+:=(u^+, u_\Gamma^+)$ and $\mathfrak{u}^- :=(u^-, u_\Gamma^-)$, where $v^+ :=\max\{v,0\}$ and $v^- :=-\min\{v,0\}$ either in $\Omega$ or on $\Gamma$. For the regularity of the solution, we introduce the following spaces
\begin{align*}
\mathbb{E}_1(t_0,t_1) &:=H^1\left(t_0,t_1 ;\mathbb{L}^2\right) \cap L^2\left(t_0,t_1 ;\mathbb{H}^2\right),\\
\mathbb{E}_1 &:= \mathbb{E}_1(0,T).
\end{align*}

System \eqref{peq1to6} is wellposed since the operator $\mathcal{A} \colon D(\mathcal{A}) \subset \mathbb{L}^2 \times \mathbb{L}^2 \rightarrow \mathbb{L}^2 \times \mathbb{L}^2$ given by $\mathcal{A}=\mathrm{diag}\left(\mathcal{A}_1, \mathcal{A}_2\right)$ and $D(\mathcal{A})=D(\mathcal{A}_1)\times D(\mathcal{A}_2)$, where
\begin{equation*}
\mathcal{A}_k=\begin{pmatrix} \dv\left(A_k\nabla\right) & 0\\ -\partial^{A_k}_\nu & \dv_\Gamma\left(D_k\nabla_\Gamma\right) \end{pmatrix}, \qquad \qquad D(\mathcal{A}_k)=\mathbb{H}^2, \quad k=1,2. \label{E21}
\end{equation*}
generates an analytic $C_0$-semigroup on $\mathbb{L}^2 \times \mathbb{L}^2$. We refer to \cite{ACMO'20} for more details. 

We will use the following assumption to prove that \eqref{peq1to6} has nonnegative solution for nonnegative initial data:

(\textbf{QP}) The functions $f_1, f_2, g_1$ and $g_2$ are quasi-positive. That is,
\begin{align*}
f_1(0,v) \ge 0 \quad \text{ and }\quad g_1(0,v) \ge 0 \quad \forall v\ge 0,\\
f_2(u,0) \ge 0 \quad \text{ and }\quad g_2(u,0) \ge 0 \quad \forall u\ge 0.
\end{align*}
Following \cite{MS'16}, we prove the following lemma.
\begin{lemma}
Let $(y_0, y_{0,\Gamma})$ and $(z_0, z_{0,\Gamma})$ be componentwise nonnegative initial data. Suppose that (\textbf{QP}) holds true. Then the solution $(y,y_\Gamma, z, z_\Gamma)$ of \eqref{peq1to6} is componentwise nonnegative.
\end{lemma}

\begin{proof}
Consider the solution of the system
\begin{empheq}[left =\empheqlbrace]{alignat=2}
\begin{aligned}
&\partial_t y = \dv(A_1(x) \nabla y) + f_1(y^+,z^+), &\text{in } \Omega_T , \\
&\partial_t z = \dv(A_2(x) \nabla z)  + f_2(y^+,z^+), &\text{in } \Omega_T , \\
&\partial_t y_{\Gamma}= \dv_{\Gamma} (D_1(x)\nabla_\Gamma y_{\Gamma}) -\partial_{\nu}^{A_1} y + g_1(y_\Gamma^+, z_\Gamma^+), &\text{on } \Gamma_T, \\
&\partial_t z_{\Gamma}= \dv_{\Gamma} (D_2(x)\nabla_\Gamma z_{\Gamma}) -\partial_{\nu}^{A_2} z + g_2(y_\Gamma^+, z_\Gamma^+), &\text{on } \Gamma_T, \\
& y_{\Gamma}(t,x) = y_{|\Gamma}(t,x), \qquad  z_{\Gamma}(t,x) = z_{|\Gamma}(t,x),  &\text{on } \Gamma_T, \\
&(y,y_{\Gamma})\rvert_{t=0}=(y_0, y_{0,\Gamma}), \qquad  (z,z_{\Gamma})\rvert_{t=0}=(z_0, z_{0,\Gamma}),   &\Omega\times\Gamma.  \label{p1eq1to6}
\end{aligned}
\end{empheq}
First, note that the functions $(y,z) \mapsto f_1(y^+,z^+)$ and $(y,z) \mapsto f_2(y^+,z^+)$ (resp. $(y_\Gamma, z_\Gamma) \mapsto g_1(y_\Gamma^+, z_\Gamma^+)$ and $(y_\Gamma, z_\Gamma) \mapsto g_2(y_\Gamma^+, z_\Gamma^+)$) are Lipschitzian. Multiplying the first equation  by $y^-$ and the third equation by $y_\Gamma^-$, we obtain
\begin{align}
y^- \partial_t y - y^-\dv(A_1(x) \nabla y) &= y^- f_1(y^+,z^+) \label{pp1}\\
y_\Gamma^- \partial_t y_{\Gamma} - y_\Gamma^- \dv_{\Gamma} (D_1(x)\nabla_\Gamma y_{\Gamma}) + y_\Gamma^- \partial_{\nu}^{A_1} y &=y_\Gamma^- g_1(y_\Gamma^+, z_\Gamma^+) \label{pp2}.
\end{align}
Since $u^- \partial_t u =-\frac{1}{2} \partial_t (u^-)^2$, integrating \eqref{pp1} and \eqref{pp2} by parts over $\Omega$ and $\Gamma$ respectively, and adding up the resulting identities, we derive
\begin{align*}
& \frac{1}{2} \partial_t \|y^-(t,\cdot)\|^2_{L^2(\Omega)} + \frac{1}{2} \partial_t \|y_\Gamma^-(t,\cdot)\|^2_{L^2(\Gamma)} + \int_\Omega A_1(x) \nabla y^- \cdot \nabla y^- \,\d x\\
& \qquad\qquad  + \int_\Gamma \langle D_1(x)\nabla_\Gamma y_\Gamma^-, \nabla_\Gamma y_\Gamma^- \rangle_\Gamma \, \d S\\
& \quad = -\int_\Omega y^- f_1(y^+,z^+)\, \d x - \int_\Gamma y_\Gamma^- g_1(y_\Gamma^+, z_\Gamma^+) \, \d S.
\end{align*}
Using the ellipticity of $A_1$ and $D_1$ with the fact that $f_1$ and $g_1$ are quasi-positive, we deduce that
\begin{equation}
\frac{1}{2} \partial_t \|y^-(t,\cdot)\|^2_{L^2(\Omega)} + \frac{1}{2} \partial_t \|y_\Gamma^-(t,\cdot)\|^2_{L^2(\Gamma)} \le 0. \label{pp3}
\end{equation}
Since $(y_0, y_{0,\Gamma})$ is componentwise nonnegative, the monotonicity inequality \eqref{pp3} implies that $(y,y_\Gamma)$ is componentwise nonnegative. Similarly, using the second and fourth equations, we obtain that $(z,z_\Gamma)$ is componentwise nonnegative.
\end{proof}

\section{Carleman estimates}\label{sec3}
Let $\eta^0 \in C^2(\overline{\Omega})$ be a function such that (see \cite{FI'96}):
\begin{empheq}[left = \empheqlbrace]{alignat=2}
\begin{aligned}
&\eta^0 > 0 \quad\text{ in } \Omega, \quad |\nabla\eta^0| \geq C_0 > 0 \quad\text{ in } \overline{\Omega\backslash\omega'},\\
&\eta^0=0 \quad\text{ and }\quad \;\partial_\nu \eta^0 \leq -c <0 \;\quad\text{ on } \Gamma \label{w2}
\end{aligned}
\end{empheq}
for some nonempty open sets $\omega' \Subset \omega \Subset \Omega$ and some positive constants $C_0, c >0$. Since $\eta^0\rvert_\Gamma=0$, we further have
\begin{equation}\label{eqgn}
\nabla \eta^0 =(\partial_\nu \eta^0) \nu\qquad \text{ on } \Gamma. 
\end{equation}
Henceforth, $A\in C^1\left(\overline{\Omega}; \mathbb{R}^{N\times N}\right)$ and $D \in C^1\left(\Gamma; \mathbb{R}^{N\times N}\right)$ designate two symmetric and uniformly elliptic matrices satisfying \eqref{uellipA} and \eqref{uellipD}. We recall the following key lemma:
\begin{lemma}[see \cite{ACMO'20}] \label{lm0} Let $\psi$ be any smooth function.
\begin{enumerate}[label=(\roman*)]
\item The following identity holds
\begin{equation}\label{eqconormal}
(\partial_\nu^A \psi)^2 -(A\nabla_\Gamma \psi\cdot \nu)^2=\left|A^{\frac{1}{2}}\nu\right|^2 \left(\left|A^{\frac{1}{2}}\nabla \psi\right|^2-\left|A^{\frac{1}{2}}\nabla_\Gamma \psi\right|^2 \right). 
\end{equation}
\item Let $c$ be the same constant in \eqref{w2}. Then
\begin{equation}\label{ceqlm1}
\partial_\nu^A \eta^0 \leq \beta \partial_\nu \eta^0 \leq -c \beta <0. 
\end{equation}
\end{enumerate}
\end{lemma}
For $t_0, t_1 \in \mathbb{R}$, $0<t_0<t_1 \leq T$, we set
$$\Omega_{t_0,t_1}:=(t_0,t_1)\times \Omega, \quad \Gamma_{t_0,t_1}:=(t_0,t_1)\times \Gamma, \quad\omega_{t_0,t_1}:=(t_0,t_1)\times \omega.$$
Now, let us consider the following weight functions
\begin{align*}
\alpha(t,x)&=\frac{\mathrm{e}^{2\lambda \|\eta^0\|_\infty}- \mathrm{e}^{\lambda \eta^0(x)}}{\gamma(t)},\\
\xi(t,x)&=\frac{\mathrm{e}^{\lambda \eta^0(x)}}{\gamma(t)}, \qquad \gamma(t)=(t-t_0)(t_1-t)
\end{align*}
for all $(t,x)\in \overline{\Omega}_{t_0, t_1}$, where $\lambda \geq 1$ is a large parameter depending only on $\Omega$ and $\omega$.

For $(z,z_\Gamma) \in \mathbb{E}_1$, we introduce the differential operators
\begin{align*}
L z &:=\partial_t z -\dv (A(x) \nabla z), &&\quad\text{ in }\Omega_T,\\
L_\Gamma(z_\Gamma, z) &:=\partial_t z_\Gamma -\dv_\Gamma (D(x)\nabla_\Gamma z_\Gamma) + \partial_\nu^A z,  &&\quad\text{ on }\Gamma_T.
\end{align*}
We set
\begin{equation*}
\sigma(x):=A(x)\nabla \eta^0(x) \cdot \nabla\eta^0(x), \qquad x \in \overline{\Omega}.
\end{equation*}
For simplicity, we will often write $z$ instead of $z_\Gamma$ on $\Gamma_T$. We also denote $Z:=(z,z_\Gamma)$ (idem for other capital letters). Let $\tau \in \mathbb{R}$ and $s>0$. Set $\psi=\mathrm{e}^{-s \alpha} \xi^{\frac{\tau}{2}} z$ and introduce the following operators
\begin{align}
M_1^{(\tau)} \psi &=  2\lambda \left(s\xi +\frac{\tau}{2}\right) A(x)\nabla \eta^0\cdot \nabla\psi +\partial_t \psi :=M_{1,1} \psi+ M_{1,2} \psi, \label{m1}\\
M_2^{(\tau)} \psi &= -\lambda^2 s^2\xi^2 \sigma \psi -\dv(A(x)\nabla \psi) + \left(\frac{\tau}{2} - s \alpha\right) (\partial_t \log \gamma)\psi\\
& :=M_{2,1} \psi+ M_{2,2} \psi + M_{2,3} \psi, \nonumber\\
N_1^{(\tau)} \psi &= \partial_t \psi - \lambda \left(s\xi +\frac{\tau}{2}\right) \partial_\nu^A \eta^0 \psi :=N_{1,1} \psi+ N_{1,2} \psi,\\
N_2^{(\tau)} \psi &= -\dv_\Gamma(D(x)\nabla_\Gamma \psi) + \left(\frac{\tau}{2}- s \alpha \right) (\partial_t \log \gamma)\psi + \partial_\nu^A \psi \label{n2}\\
& :=N_{2,1} \psi+ N_{2,2} \psi + N_{2,3} \psi. \nonumber
\end{align}
We further set
\begin{align*}
I_\Omega(\tau, z) &= \bigintssss_{\Omega_{t_0,t_1}} \mathrm{e}^{-2s\alpha} (s\xi)^{\tau-1} \left(|\partial_t z|^2 + |\dv(A(x)\nabla z)|^2 \right) \,\d x\,\d t\\
&\qquad + \lambda^2 \bigintssss_{\Omega_{t_0,t_1}} \mathrm{e}^{-2s\alpha} (s\xi)^{\tau +1} |\nabla z|^2 \,\d x\,\d t \\
& \qquad + \lambda^4 \bigintssss_{\Omega_{t_0,t_1}} \mathrm{e}^{-2s\alpha} (s\xi)^{\tau +3} |z|^2 \,\d x\,\d t ,\\
I_\Gamma(\tau, z_\Gamma, z) &= \bigintssss_{\Gamma_{t_0,t_1}} \mathrm{e}^{-2s\alpha} (s\xi)^{\tau-1} \left(|\partial_t z_\Gamma|^2 + |\dv_\Gamma(D(x)\nabla_\Gamma z_\Gamma)|^2 \right) \,\d S\,\d t\\
& + \lambda \bigintssss_{\Gamma_{t_0,t_1}} \mathrm{e}^{-2s\alpha} (s\xi)^{\tau +1} |\nabla_\Gamma z_\Gamma|^2 \,\d S\,\d t \\
& + \lambda^3 \bigintssss_{\Gamma_{t_0,t_1}} \mathrm{e}^{-2s\alpha} (s\xi)^{\tau +3} |z_\Gamma|^2 \,\d S\,\d t + \lambda \bigintssss_{\Gamma_{t_0,t_1}} \mathrm{e}^{-2s\alpha} (s\xi)^{\tau +1} |\partial_\nu^A z|^2 \,\d S\,\d t.
\end{align*}

\begin{lemma}[Carleman estimate] \label{lm2}
Let $\tau \in \mathbb{R}$. There are three positive constants $\lambda_1 =\lambda_1(\Omega,\omega), s_1=s_1(\Omega,\omega, \tau)$ and $C=C(\Omega,\omega, \tau)$ such that, for any $\lambda\geq \lambda_1$ and $s\geq s_1$, the following inequality holds
\begin{align}
I_\Omega(\tau, z) + I_\Gamma(\tau, z_\Gamma, z) &\leq C \left[\lambda^4\int_{\omega_{t_0,t_1}} \mathrm{e}^{-2s\alpha} (s\xi)^{\tau +3} |z|^2 \,\d x\,\d t \right. \nonumber\\
& \hspace{-2.7cm}\left. + \int_{\Omega_{t_0,t_1}} \mathrm{e}^{-2s\alpha} (s\xi)^{\tau} |Lz|^2 \,\d x\,\d t + \int_{\Gamma_{t_0,t_1}} \mathrm{e}^{-2s\alpha} (s\xi)^{\tau} |L_\Gamma (z_\Gamma, z)|^2 \,\d S\,\d t \right]\label{car11}
\end{align}
for all $Z=(z,z_\Gamma)\in \mathbb{E}_1$. For $\Psi=\mathrm{e}^{-s\alpha} \xi^{\frac{\tau}{2}} Z$ the inequality \eqref{car11} also yields an upper bound for $$\left\|M_1^{(\tau)} \psi\right\|_{L^2(\Omega_{t_0,t_1})}^2 + \left\|M_2^{(\tau)} \psi\right\|_{L^2(\Omega_{t_0,t_1})}^2 + \left\|N_1^{(\tau)} \psi_\Gamma\right\|_{L^2(\Gamma_{t_0,t_1})}^2 + \left\|N_2^{(\tau)} \psi_\Gamma \right\|_{L^2(\Gamma_{t_0,t_1})}^2.$$
\end{lemma}

In the case when $\tau=0$, the above lemma has been proven in \cite[Lemma 2.4]{ACMO'20} with help of Lemma \ref{lm0}. The general result for arbitrary $\tau$ will follow by taking into account the extending decomposition given by \eqref{m1}-\eqref{n2}. The proof is postponed to \ref{app}.

\begin{remark}
The Carleman estimate \eqref{car11} does not yield the desired stability estimate for the semilinear coupled system \eqref{eq1to6}, since we need appropriate powers in $s$ and $\lambda$ to absorb some terms on the right-hand side. To this end, we need a modified form of Carleman estimate with one observation.
\end{remark}
Consider the following system.
\begin{empheq}[left =\empheqlbrace]{alignat=2}
\begin{aligned}
&\partial_t y = \dv(A_1(x) \nabla y) + p_{11}(x)y + p_{12}(x)z + f_1, &\text{in } \Omega_T , \\
&\partial_t z = \dv(A_2(x) \nabla z) + p_{21}(x)y + p_{22}(x)z + f_2, &\text{in } \Omega_T , \\
&\partial_t y_{\Gamma}= \dv_{\Gamma}(D_1(x)\nabla_\Gamma y_{\Gamma}) -\partial_{\nu}^{A_1} y + q_{11}(x)y_{\Gamma} + q_{12}(x)z_{\Gamma} + g_1, &\text{on } \Gamma_T, \\
&\partial_t z_{\Gamma}= \dv_{\Gamma} (D_2(x)\nabla_\Gamma z_{\Gamma}) -\partial_{\nu}^{A_2} z + q_{21}(x)y_{\Gamma} + q_{22}(x)z_{\Gamma} + g_2, &\text{on } \Gamma_T, \\
& y_{\Gamma}(t,x) = y_{|\Gamma}(t,x), \qquad  z_{\Gamma}(t,x) = z_{|\Gamma}(t,x),  &\text{on } \Gamma_T, \\
&(y,y_{\Gamma})\rvert_{t=0}=(y_0, y_{0,\Gamma}), \qquad  (z,z_{\Gamma})\rvert_{t=0}=(z_0, z_{0,\Gamma}),   &\Omega\times\Gamma.  \label{1eq1to6}
\end{aligned}
\end{empheq}

\begin{lemma}
There exist three constants $\lambda_1=\lambda_1(\Omega, \omega) \ge 1, s_1=s_1\left(T,\lambda_1\right)>1$ and $C=C\left(\Omega, \omega, R, t_0, t_1, T, p_0\right)$ such that, for any $\lambda \ge \lambda_1$ and $s \ge s_1$ with a fixed $0<\epsilon<1$, the following inequality holds
\begin{align}
\lambda^{-4+\epsilon}[I_\Omega(-3, y)+ I_\Gamma(-3,y_\Gamma,y)] + I_\Omega(0, z) + I_\Gamma(0, z_\Gamma,z) \nonumber\\
& \hspace{-8.5cm}\leq C \left[ s^{4} \lambda^{4+\epsilon} \int_{\omega_{t_0, t_1}} \mathrm{e}^{-2 s \alpha} \xi^4 |z|^2 \,\d x\,\d t \right. \nonumber\\
& \hspace{-8 cm} + s^{-3} \lambda^{-4+\epsilon} \left(\int_{\Omega_{t_0, t_1}} \mathrm{e}^{-2 s \alpha} \xi^{-3} |f_1|^2 \,\d x\,\d t + \int_{\Gamma_{t_0, t_1}} \mathrm{e}^{-2 s \alpha} \xi^{-3} |g_1|^2 \,\d S\,\d t \right) \nonumber\\
& \hspace{-8 cm} + \left. \lambda^{2 \epsilon} \left(\int_{\Omega_{t_0, t_1}} \mathrm{e}^{-2 s \alpha} |f_2|^2 \,\d x\,\d t + \int_{\Gamma_{t_0, t_1}} \mathrm{e}^{-2 s \alpha} |g_2|^2 \,\d S\,\d t \right) \right]. \label{shcar}
\end{align}
\end{lemma}

\begin{proof}
We adapt the strategy in \cite{CGRY'12} to our case by treating several new boundary terms. Applying Carleman estimate \eqref{car11} for $\tau=-3$ to $(y,y_\Gamma)$ and for $\tau=0$ to $(z,z_\Gamma)$, we have
\begin{align}
\lambda^{-4+\epsilon}[I_\Omega(-3, y)+ I_\Gamma(-3,y_\Gamma,y)] + I_\Omega(0, z) + I_\Gamma(0, z_\Gamma,z) \nonumber\\
& \hspace{-8.5cm}\leq C \left[\lambda^{\epsilon} \int_{\omega_{t_0, t_1}} \mathrm{e}^{-2 s \alpha} |y|^2 \,\d x\,\d t + s^{4} \lambda^{4+\epsilon} \int_{\omega_{t_0, t_1}} \mathrm{e}^{-2 s \alpha} \xi^4 |z|^2 \,\d x\,\d t \right. \nonumber\\
& \hspace{-8 cm} + s^{-3} \lambda^{-4+\epsilon} \left(\int_{\Omega_{t_0, t_1}} \mathrm{e}^{-2 s \alpha} \xi^{-3} (|p_{11}y|^2+|p_{12}z|^2 +|f_1|^2) \,\d x\,\d t \right. \nonumber\\
&\hspace{-6.2 cm} \left. + \int_{\Gamma_{t_0, t_1}} \mathrm{e}^{-2 s \alpha} \xi^{-3} (|q_{11}y_\Gamma|^2+|q_{12}z_\Gamma|^2 +|g_1|^2) \,\d S\,\d t \right) \nonumber\\
& \hspace{-5.8 cm} + \int_{\Omega_{t_0, t_1}} \mathrm{e}^{-2 s \alpha} (|p_{21}y|^2+|p_{22}z|^2 +|f_2|^2) \,\d x\,\d t \nonumber\\
& \hspace{-5.8 cm} \left.  + \int_{\Gamma_{t_0, t_1}} \mathrm{e}^{-2 s \alpha} (|q_{21}y_\Gamma|^2 +|q_{22}z_\Gamma|^2 +|g_2|^2) \,\d S\,\d t \right]. \label{e1}
\end{align}
Next we estimate the first term $\mathbf{I}$ on the right hand side of \eqref{e1} by means of the second one. To this end, choose $\omega' \Subset \omega'' \Subset \omega$ and introduce a cut-off function $\chi =\chi(x)$ such that
$$\chi \in C^2(\Omega), \qquad \mathrm{supp}\, \chi \subset \omega'', \qquad \chi \equiv 1 \quad \text{ in } \omega', \qquad 0<\chi \le 1 \quad \text{ in } \overline{\omega''},$$
in order to estimate the term
$$\mathbf{I}':=\lambda^{\epsilon} \int_{\omega_{t_0, t_1}} \mathrm{e}^{-2 s \alpha} \chi p_{21} |y|^2 \,\d x\,\d t.$$
By virtue of $\eqref{1eq1to6}_2$, we have
\begin{align*}
\mathbf{I}' &= \lambda^{\epsilon} \int_{\omega_{t_0, t_1}} \mathrm{e}^{-2 s \alpha} \chi (\partial_t z -\dv(A_2(x) \nabla z) - p_{22}(x)z - f_2) y \,\d x\,\d t \\
& =\lambda^{\epsilon} \int_{\omega_{t_0, t_1}} \mathrm{e}^{-2 s \alpha} \chi (\partial_t z) y \,\d x\,\d t - \lambda^{\epsilon} \int_{\omega_{t_0, t_1}} \mathrm{e}^{-2 s \alpha} \chi \dv(A_2(x) \nabla z) y \,\d x\,\d t\\
& \quad - \lambda^{\epsilon} \int_{\omega_{t_0, t_1}} \mathrm{e}^{-2 s \alpha} \chi p_{22}(x)z y \,\d x\,\d t - \lambda^{\epsilon} \int_{\omega_{t_0, t_1}} \mathrm{e}^{-2 s \alpha} \chi f_2 y \,\d x\,\d t\\
&=: \mathbf{I}_1 + \mathbf{I}_2 + \mathbf{I}_3 +\mathbf{I}_4.
\end{align*}
We start with the integrals $\mathbf{I}_3$ and $\mathbf{I}_4$. Since $\chi, p_{22}$ and $\xi^{-1}$ are bounded, by Young's inequality for $(\frac{1}{\sqrt{\gamma}} \xi^2 \chi p_{22} z)(\sqrt{\gamma} \xi^{-2}y)$ and $(\lambda^{\epsilon} f_2)(\chi y)$, we get
\begin{align*}
|\mathbf{I}_3| &\le C \left[\frac{1}{2\gamma} s^{4} \lambda^{4+\epsilon} \int_{\omega_{t_0, t_1}} \mathrm{e}^{-2 s \alpha} \xi^4 |z|^2 \,\d x\,\d t + \frac{\gamma}{2} \lambda^\epsilon \int_{\Omega_{t_0, t_1}} \mathrm{e}^{-2 s \alpha} |y|^2 \,\d x\,\d t\right], \\
|\mathbf{I}_4| &\le C \left[\lambda^{2 \epsilon} \int_{\omega_{t_0, t_1}} \mathrm{e}^{-2 s \alpha} |f_2|^2 \,\d x\,\d t + \int_{\omega_{t_0, t_1}} \mathrm{e}^{-2 s \alpha} |y|^2 \,\d x\,\d t\right]
\end{align*}
for all $\gamma>0$. The second terms in the above inequalities can be absorbed into the left-hand side thanks to the third term of $\lambda^{-4+\epsilon} I_\Omega(-3, y)$ and $\lambda$ large enough. Therefore,
\begin{align*}
|\mathbf{I}_3| &\le C s^{4} \lambda^{4+\epsilon} \int_{\omega_{t_0, t_1}} \mathrm{e}^{-2 s \alpha} \xi^4 |z|^2 \,\d x\,\d t + \frac{1}{2} \lambda^{-4+\epsilon} I_\Omega(-3, y), \\
|\mathbf{I}_4| &\le C \lambda^{2 \epsilon} \int_{\Omega_{t_0, t_1}} \mathrm{e}^{-2 s \alpha} |f_2|^2 \,\d x\,\d t + \frac{1}{2} \lambda^{-4+\epsilon} I_\Omega(-3, y).
\end{align*}
Next we estimate $\mathbf{I}_1$. Integrating by parts with respect to $t$ we obtain
$$\mathbf{I}_1=-\lambda^{\epsilon} \int_{\omega_{t_0, t_1}} \mathrm{e}^{-2 s \alpha} \chi z \partial_t y \,\d x\,\d t + 2 s \lambda^{\epsilon} \int_{\omega_{t_0, t_1}} \mathrm{e}^{-2 s \alpha} \chi (\partial_t \alpha) z y \,\d x\,\d t =: \mathbf{I}_{11}+\mathbf{I}_{12}.$$
By Young's inequality for $(\frac{1}{\sqrt{\gamma}}s^2 \lambda^2 \xi^2 \chi z)(\sqrt{\gamma} s^{-2} \lambda^{-2} \xi^{-2}\partial_t y)$, we derive
\begin{align*}
|\mathbf{I}_{11}| &\le \frac{1}{2\gamma} s^{4} \lambda^{4+\epsilon} \int_{\omega_{t_0, t_1}} \mathrm{e}^{-2 s \alpha} \xi^4 \chi^2 |z|^2 \,\d x\,\d t + \frac{\gamma}{2} s^{-4}\lambda^{-4+\epsilon} \int_{\Omega_{t_0, t_1}} \mathrm{e}^{-2 s \alpha} \xi^{-4} |\partial_t y|^2 \,\d x\,\d t \\
& \hspace{-0.5cm} \le C \left[\frac{1}{2\gamma} s^{4} \lambda^{4+\epsilon} \int_{\omega_{t_0, t_1}} \mathrm{e}^{-2 s \alpha} \xi^4 |z|^2 \,\d x\,\d t + \frac{\gamma}{2} s^{-4}\lambda^{-4+\epsilon} \int_{\Omega_{t_0, t_1}} \mathrm{e}^{-2 s \alpha} \xi^{-4} |\partial_t y|^2 \,\d x\,\d t\right].
\end{align*}
The last term of the above inequality can be absorbed into the left-hand side by the first term of $\lambda^{-4+\epsilon} I_\Omega(-3, y)$ by choosing $\gamma>0$ small enough. Moreover, using $|\partial_t \alpha|^2 \le C \xi^4$ and applying Young's inequality to $(s^{1/2} \lambda^{1/2} \chi |\partial_t \alpha| z)(s^{-1/2} \lambda^{-1/2} y)$, we obtain
\begin{align*}
|\mathbf{I}_{12}| &\le s\lambda^\epsilon \left[ s\lambda \int_{\omega_{t_0, t_1}} \mathrm{e}^{-2 s \alpha} \chi^2 |\partial_t \alpha|^2 |z|^2 \,\d x\,\d t + s^{-1}\lambda^{-1}\int_{\omega_{t_0, t_1}} \mathrm{e}^{-2 s \alpha} |y|^2 \,\d x\,\d t\right]\\
& \le C \left[ s^4 \lambda^{4+\epsilon} \int_{\omega_{t_0, t_1}} \mathrm{e}^{-2 s \alpha} \xi^4 |z|^2 \,\d x\,\d t + \lambda^{\epsilon-1}\int_{\omega_{t_0, t_1}} \mathrm{e}^{-2 s \alpha} |y|^2 \,\d x\,\d t\right].
\end{align*}
The last term of the above inequality can be absorbed into the left-hand side by the first term of $\lambda^{-4+\epsilon} I_\Omega(-3, y)$, since $\lambda^{\epsilon-1} \le \lambda^\epsilon$ ($\lambda \ge 1$). Hence,
$$|\mathbf{I}_{1}| \le C s^4 \lambda^{4+\epsilon} \int_{\omega_{t_0, t_1}} \mathrm{e}^{-2 s \alpha} \xi^4 |z|^2 \,\d x\,\d t + \frac{1}{2} \lambda^{-4+\epsilon} I_\Omega(-3, y).$$
To estimate $\mathbf{I}_2$, we use integration by parts in space to obtain
\begin{align*}
\mathbf{I}_2 &= - \lambda^{\epsilon} \int_{\omega_{t_0, t_1}} \dv(A_2(x) \nabla(\mathrm{e}^{-2 s \alpha} \chi y)) z \,\d x\,\d t \\
& \hspace{-0.3cm} = - \lambda^{\epsilon} \left[\int_{\omega_{t_0, t_1}} \mathrm{e}^{-2 s \alpha} \chi \dv(A_2(x) \nabla y) z \,\d x\,\d t + 2 s\lambda \int_{\omega_{t_0, t_1}} \mathrm{e}^{-2 s \alpha} \xi \dv(A_2(x) \nabla \eta^0)\chi y z \,\d x\,\d t \right.\\
& \hspace{1.4cm} \left. + \int_{\omega_{t_0, t_1}} \mathrm{e}^{-2 s \alpha} \dv(A_2(x) \nabla \chi) y z \,\d x\,\d t\right]\\
& \hspace{-0.3cm} \le C \left[\frac{1}{2\gamma} s^{4} \lambda^{4+\epsilon} \int_{\omega_{t_0, t_1}} \mathrm{e}^{-2 s \alpha} \xi^4 |z|^2 \,\d x\,\d t + \frac{\gamma}{2} s^{-4}\lambda^{-4+\epsilon} \int_{\Omega_{t_0, t_1}} \mathrm{e}^{-2 s \alpha} \xi^{-4} |\dv(A_2(x) \nabla y)|^2 \,\d x\,\d t\right.\\
& \hspace{1cm} \left. + \frac{\gamma}{2} \lambda^{\epsilon} \int_{\Omega_{t_0, t_1}} \mathrm{e}^{-2 s \alpha} |y|^2 \,\d x\,\d t\right].
\end{align*}
The last two terms in the above inequality can be absorbed into the left-hand side by first and third terms of $\lambda^{-4+\epsilon} I_\Omega(-3, y)$ respectively by choosing $\gamma$ small. Thus,
$$|\mathbf{I}_{2}| \le C s^4 \lambda^{4+\epsilon} \int_{\omega_{t_0, t_1}} \mathrm{e}^{-2 s \alpha} \xi^4 |z|^2 \,\d x\,\d t + \frac{1}{2} \lambda^{-4+\epsilon} I_\Omega(-3, y).$$
We deduce that
$$|\mathbf{I}_1|+|\mathbf{I}_2|+|\mathbf{I}_3| \le C s^{4} \lambda^{4+\epsilon} \int_{\omega_{t_0, t_1}} \mathrm{e}^{-2 s \alpha} \xi^4 |z|^2 \,\d x\,\d t + \frac{1}{2} \lambda^{-4+\epsilon} I_\Omega(-3, y).$$
Therefore, since $p_{21} \ge p_0 >0$ and $\chi_0 :=\min\limits_{\overline{\omega''}}\chi >0$, we obtain
\begin{align*}
|\mathbf{I}| &\le \frac{1}{p_0 \chi_0} |\mathbf{I}'| \\
& \hspace{-0.3cm}\le C s^{4} \lambda^{4+\epsilon} \int_{\omega_{t_0, t_1}} \mathrm{e}^{-2 s \alpha} \xi^4 |z|^2 \,\d x\,\d t + C\lambda^{2 \epsilon} \int_{\omega_{t_0, t_1}} \mathrm{e}^{-2 s \alpha} |f_2|^2 + \frac{1}{2} \lambda^{-4+\epsilon} I_\Omega(-3, y)\,\d x\,\d t .
\end{align*}

Now we estimate the terms that contain potentials in \eqref{e1}. Since $p_{11}, p_{12}, q_{11}, q_{12}, \xi^{-1}$ are bounded, $\xi \ge c>0$, $s^{-3} \lambda^{-4+\epsilon} \le \frac{\lambda^{\epsilon-1}}{2}$, $\lambda^{-4+\epsilon} \le \lambda^3$ for $\epsilon<1$, $s\ge 1$ and $\lambda$ large, then we obtain
\begin{align*}
&s^{-3} \lambda^{-4+\epsilon} \int_{\Omega_{t_0, t_1}} \mathrm{e}^{-2 s \alpha} \xi^{-3} (|p_{11}y|^2+|p_{12}z|^2) \,\d x\,\d t \\
& + s^{-3} \lambda^{-4+\epsilon} \int_{\Gamma_{t_0, t_1}} \mathrm{e}^{-2 s \alpha} \xi^{-3} (|q_{11}y_\Gamma|^2+|q_{12}z_\Gamma|^2) \,\d S\,\d t\\
& \hspace{-0.5cm}\le \frac{\lambda^{\epsilon-1}}{2} \left(\int_{\Omega_{t_0, t_1}} \mathrm{e}^{-2 s \alpha} |y|^2 \,\d x\,\d t + \int_{\Gamma_{t_0, t_1}} \mathrm{e}^{-2 s \alpha} |y_\Gamma|^2 \,\d S\,\d t \right)\\
& + \frac{s^3 \lambda^3}{2} \left(\int_{\Omega_{t_0, t_1}} \mathrm{e}^{-2 s \alpha} \xi^{3} |z|^2 \,\d x\,\d t + \int_{\Gamma_{t_0, t_1}} \mathrm{e}^{-2 s \alpha} \xi^{3} |z_\Gamma|^2 \,\d S\,\d t\right)\\
& \le \frac{1}{2} \left(\lambda^{-4+\epsilon} [I_\Omega(-3, y)+ I_\Gamma(-3, y_\Gamma, y)] + I_\Omega(0, z)+ I_\Gamma(0, z_\Gamma, z)\right).
\end{align*}
The right-hand side term in the above inequality can be absorbed by the left-hand side of \eqref{e1}. Finally, since the potentials are bounded, 
\begin{align*}
&\int_{\Omega_{t_0, t_1}} \mathrm{e}^{-2 s \alpha} (|p_{21}y|^2+|p_{22}z|^2) \,\d x\,\d t \nonumber\\
& + \int_{\Gamma_{t_0, t_1}} \mathrm{e}^{-2 s \alpha} (|q_{21}y_\Gamma|^2 +|q_{22}z_\Gamma|^2) \,\d S\,\d t\\
&  \hspace{-0.5cm}\le C \left(\int_{\Omega_{t_0, t_1}} \mathrm{e}^{-2 s \alpha} |y|^2 \,\d x\,\d t + \int_{\Gamma_{t_0, t_1}} \mathrm{e}^{-2 s \alpha} |y_\Gamma|^2 \,\d S\,\d t \right.\\
& \left. \qquad + \int_{\Omega_{t_0, t_1}} \mathrm{e}^{-2 s \alpha} |z|^2 \,\d x\,\d t + \int_{\Gamma_{t_0, t_1}} \mathrm{e}^{-2 s \alpha} |z_\Gamma|^2 \,\d S\,\d t \right)\\
& \le \frac{1}{2} \left(\lambda^{-4+\epsilon} [I_\Omega(-3, y)+ I_\Gamma(-3, y_\Gamma, y)] + I_\Omega(0, z)+ I_\Gamma(0, z_\Gamma, z)\right),
\end{align*}
where we have chosen $\lambda$ and $s$ large so that $C\le \frac{\lambda^{\epsilon-1}}{2}$ and $C\le \frac{s^3\lambda^4 \xi^3}{2}$. Then the last two terms with potentials in \eqref{e1} are then absorbed by the left-hand side. This achieves the proof.
\end{proof}

\section{Proof of stability estimate}\label{sec4}
Now we are ready to establish the Lipschitz stability estimate for our coefficients inverse problem.
  
\begin{proof}[Proof of Theorem \ref{thm1}]
Let us set $$U:=(u,u_\Gamma)=\partial_t (Y-\widetilde{Y}) \quad \text{ and } \quad  V:=(v,v_\Gamma)=\partial_t (Z-\widetilde{Z}),$$
where $(Y,Z)$ is the solution of \eqref{eq1to6} and $(\widetilde{Y},\widetilde{Z})$ is the solution corresponding to $(\widetilde{p}_{13}, \widetilde{p}_{21}, \widetilde{q}_{13}, \widetilde{q}_{21}, \widetilde{y}_0, \widetilde{y}_{0,\Gamma}, \widetilde{z}_0, \widetilde{z}_{0,\Gamma})$. Subtracting the system satisfied by $(\widetilde{Y},\widetilde{Z})$ from the system satisfied by $(Y,Z)$ and taking the time derivative in the resulting system, we obtain
\begin{small}
\begin{empheq}[left =\empheqlbrace]{alignat=2}
\begin{aligned}
&\partial_t u = \dv(A_1\nabla u) + p_{11}u + p_{12}v + a_1 \partial_t f(\widetilde{y},\widetilde{z}) + p_{13} \partial_t F(y,z,\widetilde{y},\widetilde{z}) &\text{in } \Omega_T , \\
&\partial_t v = \dv(A_2\nabla v) + p_{21}u + p_{22}v + a_2 \partial_t \widetilde{y} &\text{in } \Omega_T , \\
& \partial_t u_{\Gamma}= \dv_{\Gamma}(D_1\nabla_\Gamma u_{\Gamma}) -\partial_{\nu}^{A_1} u + q_{11}u_{\Gamma} + q_{12}v_{\Gamma} +\ell_1 \partial_t g(\widetilde{y}_\Gamma, \widetilde{z}_\Gamma)\\
& \hspace{1cm} + q_{13} \partial_t G(y_{\Gamma},z_{\Gamma}, \widetilde{y}_\Gamma, \widetilde{z}_\Gamma) &\text{on } \Gamma_T, \\
&\partial_t v_{\Gamma}= \dv_{\Gamma} (D_2\nabla_\Gamma v_{\Gamma}) -\partial_{\nu}^{A_2} v + q_{21} u_{\Gamma} + q_{22}v_{\Gamma} + \ell_2 \partial_t \widetilde{y}_\Gamma &\text{on } \Gamma_T, \\
&u_{\Gamma}(t,x) = u_{|\Gamma}(t,x), \qquad  v_{\Gamma}(t,x) = v_{|\Gamma}(t,x)  &\text{on } \Gamma_T, \label{2eq1to6}
\end{aligned}
\end{empheq}
\end{small}
where $a_1 := p_{13}-\widetilde{p}_{13},\; a_2 :=p_{21}-\widetilde{p}_{21},\; \ell_1 := q_{13}-\widetilde{q}_{13},\; \ell_2 := q_{21}-\widetilde{q}_{21}$ and $F(y,z,\widetilde{y},\widetilde{z}):=f(y,z)-f(\widetilde{y},\widetilde{z}),\; G(y_{\Gamma},z_{\Gamma}, \widetilde{y}_\Gamma, \widetilde{z}_\Gamma):= g(y_{\Gamma},z_{\Gamma})-g(\widetilde{y}_\Gamma,\widetilde{z}_\Gamma)$. The Carleman estimate \eqref{shcar} applied to \eqref{2eq1to6} yields
\begin{align}
\lambda^{-4+\epsilon}[I_\Omega(-3, u)+ I_\Gamma(-3,u_\Gamma,u)] + I_\Omega(0, v) + I_\Gamma(0, v_\Gamma,v) \nonumber\\
& \hspace{-8.5cm}\leq C \left[ s^{4} \lambda^{4+\epsilon} \int_{\omega_{t_0, t_1}} \mathrm{e}^{-2 s \alpha} \xi^4 |v|^2 \,\d x\,\d t \right. \nonumber\\
& \hspace{-7.8 cm} + s^{-3} \lambda^{-4+\epsilon} \left(\int_{\Omega_{t_0, t_1}} \mathrm{e}^{-2 s \alpha} \xi^{-3} (|a_1 \partial_t f(\widetilde{y},\widetilde{z})|^2 + |\partial_t F|^2) \,\d x\,\d t \right.\nonumber \\
& \hspace{-5.8 cm} + \left. \int_{\Gamma_{t_0, t_1}} \mathrm{e}^{-2 s \alpha} \xi^{-3} (|\ell_1 \partial_t g(\widetilde{y}_\Gamma, \widetilde{z}_\Gamma)|^2 + |\partial_t G|^2) \,\d S\,\d t \right) \nonumber\\
& \hspace{-7.5 cm} + \left. \lambda^{2 \epsilon} \left(\int_{\Omega_{t_0, t_1}} \mathrm{e}^{-2 s \alpha} |a_2 \partial_t \widetilde{y}|^2 \,\d x\,\d t + \int_{\Gamma_{t_0, t_1}} \mathrm{e}^{-2 s \alpha} |\ell_2 \partial_t\widetilde{y}_\Gamma|^2 \,\d S\,\d t \right) \right]. \label{e2}
\end{align}
Next we control the term
$$\mathbf{A}:=s^{-3} \lambda^{-4+\epsilon} \left(\int_{\Omega_{t_0, t_1}} \mathrm{e}^{-2 s \alpha} \xi^{-3} |\partial_t F|^2 \,\d x\,\d t + \int_{\Gamma_{t_0, t_1}} \mathrm{e}^{-2 s \alpha} \xi^{-3} |\partial_t G|^2 \,\d S\,\d t \right).$$
Using Assumption II-(i), the estimate $\xi^{-3} \le C \xi^3$ and the boundedness of $\xi^{-1}$, we obtain
$$\mathbf{A} \le C s^{-3} \lambda^{-4+\epsilon} \left(\int_{\Omega_{t_0, t_1}} \mathrm{e}^{-2 s \alpha} (|u|^2+\xi^{3} |v|^2) \,\d x\,\d t + \int_{\Gamma_{t_0, t_1}} \mathrm{e}^{-2 s \alpha} (|u_\Gamma|^2+\xi^{3} |v_\Gamma|^2) \,\d S\,\d t \right).$$
Increasing $s$, one can absorb $\mathbf{A}$ by the left hand side of \eqref{e2}. The latter becomes
\begin{align}
&\lambda^{-4+\epsilon}[I_\Omega(-3, u)+ I_\Gamma(-3,u_\Gamma,u)] + I_\Omega(0, v) + I_\Gamma(0, v_\Gamma,v) \nonumber\\
& \leq C \left[ s^{4} \lambda^{4+\epsilon} \int_{\omega_{t_0, t_1}} \mathrm{e}^{-2 s \alpha} \xi^4 |v|^2 \,\d x\,\d t \right. + s^{-3} \lambda^{-4+\epsilon} \left(\int_{\Omega_{t_0, t_1}} \mathrm{e}^{-2 s \alpha} \xi^{-3} |a_1 \partial_t f(\widetilde{y},\widetilde{z})|^2 \,\d x\,\d t \right.\nonumber \\
& \hspace{3cm} + \left. \int_{\Gamma_{t_0, t_1}} \mathrm{e}^{-2 s \alpha} \xi^{-3} |\ell_1 \partial_t g(\widetilde{y}_\Gamma, \widetilde{z}_\Gamma)|^2 \,\d S\,\d t \right) \nonumber\\
& \hspace{1cm} + \left. \lambda^{2 \epsilon} \left(\int_{\Omega_{t_0, t_1}} \mathrm{e}^{-2 s \alpha} |a_2 \partial_t \widetilde{y}|^2 \,\d x\,\d t + \int_{\Gamma_{t_0, t_1}} \mathrm{e}^{-2 s \alpha} |\ell_2 \partial_t\widetilde{y}_\Gamma|^2 \,\d S\,\d t \right) \right]. \label{e3}
\end{align}
Now, set $\Psi :=(\psi,\psi_\Gamma)= \mathrm{e}^{-s\alpha} \xi^{-3/2} U$ and
$$\mathbf{I}_1 := \lambda^{-4+\epsilon} \int_{\Omega_{t_0,\theta}} M_1^{(-3)} \psi\, \psi \,\d x\,\d t, \qquad \mathbf{J}_1 := \lambda^{-4+\epsilon} \int_{\Gamma_{t_0,\theta}} N_1^{(-3)} \psi_\Gamma\, \psi_\Gamma \,\d S\,\d t.$$
Applying Cauchy-Schwarz inequality and then Young inequality, we obtain
\begin{align}
|\mathbf{I}_1| &\leq \lambda^{-2}\left(\lambda^{-4+\epsilon}\int_{\Omega_{t_0, \theta}} \left(M_1^{(-3)} \psi\right)^2 \,\d x\,\d t \right )^{1/2} \left(\lambda ^{\epsilon} \int_{\Omega_{t_0,\theta}} \mathrm{e}^{-2s\alpha} \xi^{-3} u^2 \,\d x\,\d t \right)^{1/2} \nonumber\\
&\leq \frac{1}{2} \lambda^{-2} \left(\lambda^{-4+\epsilon} \| M_1^{(-3)} \psi\|^2_{L^2(\Omega_{t_0, t_1})}+  \lambda^\epsilon \int_{\Omega_{t_0, t_1}} \mathrm{e}^{-2s\alpha} \xi^{-3} u^2 \,\d x\,\d t \right)\label{el2s1}
\end{align}
for $\lambda \ge 1$. Using again Cauchy-Schwarz and Young inequalities, we infer that
\begin{align}
|\mathbf{J}_1| &\leq \lambda^{-2}\left(\lambda^{-4+\epsilon}\int_{\Gamma_{t_0, \theta}} \left(N_1^{(-3)} \psi_\Gamma\right)^2 \,\d S\,\d t \right )^{1/2} \left(\lambda ^{\epsilon} \int_{\Gamma_{t_0,\theta}} \mathrm{e}^{-2s\alpha} \xi^{-3} u_\Gamma^2 \,\d S\,\d t \right)^{1/2} \nonumber\\
& \le \frac{1}{2} \lambda^{-2} \left(\lambda^{-4+\epsilon} \| N_1^{(-3)} \psi_\Gamma\|^2_{L^2(\Gamma_{t_0, t_1})}+ \lambda^\epsilon \int_{\Gamma_{t_0, t_1}} \mathrm{e}^{-2s\alpha} \xi^{-3} u_\Gamma^2 \,\d S\,\d t \right).\label{el4s1}
\end{align}
By adding \eqref{el2s1} to \eqref{el4s1} and using the shifted Carleman estimate \eqref{e3}, we find
\begin{align}
|\mathbf{I}_1|+|\mathbf{J}_1| \le \frac{1}{2} \lambda^{-2} C \left[s^{4} \lambda^{4+\epsilon} \int_{\omega_{t_0, t_1}} \mathrm{e}^{-2 s \alpha} \xi^4 |v|^2 \,\d x\,\d t \right. \nonumber\\
& \hspace{-9cm}+  s^{-3} \lambda^{-4+\epsilon} \left(\int_{\Omega_{t_0, t_1}} \mathrm{e}^{-2 s \alpha} \xi^{-3} |a_1 \partial_t f(\widetilde{y},\widetilde{z})|^2 \,\d x\,\d t + \int_{\Gamma_{t_0, t_1}} \mathrm{e}^{-2 s \alpha} \xi^{-3} |\ell_1 \partial_t g(\widetilde{y}_\Gamma, \widetilde{z}_\Gamma)|^2 \,\d S\,\d t \right) \nonumber \\
& \hspace{-6cm} \left. + \lambda^{2 \epsilon} \left(\int_{\Omega_{t_0, t_1}} \mathrm{e}^{-2 s \alpha} |a_2 \partial_t \widetilde{y}|^2 \,\d x\,\d t + \int_{\Gamma_{t_0, t_1}} \mathrm{e}^{-2 s \alpha} |\ell_2 \partial_t\widetilde{y}_\Gamma|^2 \,\d S\,\d t \right) \right].
\end{align}
Also, by computing $\mathbf{I}_1$ using integration by parts over $\Omega$, we obtain
\begin{align*}
\mathbf{I}_1 &=\lambda^{-4+\epsilon} \left[\frac{1}{2} \int_{\Omega_{t_0, \theta}} \partial_t (\psi^2) \,\d x\,\d t -5 s\lambda^2 \int_{\Omega_{t_0, \theta}} \xi \sigma \psi^2 \,\d x\,\d t \right.\\
&\hspace{2cm} - \lambda \int_{\Omega_{t_0, \theta}} \left(s\xi -\frac{3}{2}\right) \dv(A_1(x)\nabla \eta^0) \psi^2 \,\d x\,\d t \\
&\hspace{2cm}  + \left. \lambda \int_{\Gamma_{t_0, \theta}} \left(s\xi -\frac{3}{2}\right) \partial_\nu^{A_1} \eta^0 \psi_\Gamma^2 \,\d S\,\d t \right].
\end{align*}
Since $\psi(t_0, x)=\lim\limits_{t \to t_0} \mathrm{e}^{-s\alpha(t,x)} \xi^{-3/2}(t,x) u(t,x)=0$, we obtain
\begin{align*}
\lambda^{-4+\epsilon} \int_\Omega |\psi(\theta,\cdot)|^2 \,\d x &= 2\mathbf{I}_1 + 2\lambda^{-4+\epsilon} \left[5 s\lambda^2 \int_{\Omega_{t_0, \theta}} \xi \sigma \psi^2 \,\d x\,\d t \right.\\
& \hspace{-3cm}+ \lambda \int_{\Omega_{t_0, \theta}} \left(s\xi -\frac{3}{2}\right) \dv(A_1(x)\nabla \eta^0) \psi^2 \,\d x\,\d t \left. - \lambda \int_{\Gamma_{t_0, \theta}} \left(s\xi -\frac{3}{2}\right) \partial_\nu^{A_1} \eta^0 \psi_\Gamma^2 \,\d S\,\d t \right].
\end{align*}
On the other hand,
\begin{align*}
\mathbf{J}_1 &=\lambda^{-4+\epsilon} \left[\frac{1}{2} \int_{\Gamma_{t_0, \theta}} \partial_t (\psi_\Gamma^2) \,\d S\,\d t - \lambda \int_{\Gamma_{t_0, \theta}} \left(s\xi -\frac{3}{2}\right) \partial_\nu^{A_1} \eta^0 \psi_\Gamma^2 \,\d S\,\d t \right].
\end{align*}
Using $\psi_\Gamma(t_0, x)=\lim\limits_{t \to t_0} \mathrm{e}^{-s\alpha(t,x)} \xi^{-3/2}(t,x) u_\Gamma(t,x)=0$, we get
$$\lambda^{-4+\epsilon} \int_\Gamma |\psi_\Gamma(\theta,\cdot)|^2 \,\d S = 2\mathbf{J}_1 + 2\lambda^{-3+\epsilon} \int_{\Gamma_{t_0, \theta}} \left(s\xi -\frac{3}{2}\right) (\partial_\nu^{A_1} \eta^0) \psi_\Gamma^2 \,\d S\,\d t.$$
By making use of the estimate $\xi^{-1}\le C$, we get
\begin{align}
\lambda^{-4+\epsilon} \left(\int_\Omega |\psi(\theta,\cdot)|^2 \,\d x + \int_\Gamma |\psi_\Gamma(\theta,\cdot)|^2 \,\d S \right) \nonumber\\
& \hspace{-6.5cm}\le 2 (|\mathbf{I}_1|+|\mathbf{J}_1|) + C \lambda^{-3+\epsilon} s (\lambda +1) \int_{\Omega_{t_0,\theta}} \mathrm{e}^{-2s\alpha} \xi^{-2} |u|^2 \,\d x\, \d t \nonumber\\
& \hspace{-6.2cm} + C \lambda^{-3+\epsilon} s \int_{\Gamma_{t_0,\theta}} \mathrm{e}^{-2s\alpha} \xi^{-2} |u_\Gamma|^2 \,\d S\, \d t \nonumber\\
& \hspace{-6.5cm} \leq C s\lambda^{-2}\left[ s^{4} \lambda^{4+\epsilon} \int_{\omega_{t_0, t_1}} \mathrm{e}^{-2 s \alpha} \xi^4 |v|^2 \,\d x\,\d t\nonumber \right.\\
& \hspace{-6.3cm} + s^{-3} \lambda^{-4+\epsilon} \left(\int_{\Omega_{t_0, t_1}} \mathrm{e}^{-2 s \alpha} \xi^{-3} |a_1 \partial_t f(\widetilde{y},\widetilde{z})|^2 \,\d x\,\d t \right.\nonumber \\
& \hspace{-6.3cm} + \left. \int_{\Gamma_{t_0, t_1}} \mathrm{e}^{-2 s \alpha} \xi^{-3} |\ell_1 \partial_t g(\widetilde{y}_\Gamma, \widetilde{z}_\Gamma)|^2 \,\d S\,\d t \right) \nonumber\\
& \hspace{-6.3cm} + \left. \lambda^{2 \epsilon} \left(\int_{\Omega_{t_0, t_1}} \mathrm{e}^{-2 s \alpha} |a_2 \partial_t \widetilde{y}|^2 \,\d x\,\d t + \int_{\Gamma_{t_0, t_1}} \mathrm{e}^{-2 s \alpha} |\ell_2 \partial_t\widetilde{y}_\Gamma|^2 \,\d S\,\d t \right) \right]. \label{e4}
\end{align}
Therefore,
\begin{align}
\lambda^{-4+\epsilon} \left(\int_\Omega \mathrm{e}^{-2 s \alpha(\theta,\cdot)} \xi^{-3}(\theta,\cdot)|u(\theta,\cdot)|^2 \,\d x + \int_\Gamma \mathrm{e}^{-2 s \alpha(\theta,\cdot)} \xi^{-3}(\theta,\cdot) |u_\Gamma(\theta,\cdot)|^2 \,\d S \right) \nonumber\\
& \hspace{-12cm} \leq C \left[ s^{5} \lambda^{2+\epsilon} \int_{\omega_{t_0, t_1}} \mathrm{e}^{-2 s \alpha} \xi^4 |v|^2 \,\d x\,\d t \right. \nonumber\\
& \hspace{-11cm} + s^{-2} \lambda^{-6+\epsilon} \left(\int_{\Omega_{t_0, t_1}} \mathrm{e}^{-2 s \alpha} \xi^{-3} |a_1 \partial_t f(\widetilde{y},\widetilde{z})|^2 \,\d x\,\d t \right.\nonumber \\
& \hspace{-11cm} + \left. \int_{\Gamma_{t_0, t_1}} \mathrm{e}^{-2 s \alpha} \xi^{-3} |\ell_1 \partial_t g(\widetilde{y}_\Gamma, \widetilde{z}_\Gamma)|^2 \,\d S\,\d t \right) \nonumber\\
& \hspace{-11cm} + \left. s\lambda^{-2+2 \epsilon} \left(\int_{\Omega_{t_0, t_1}} \mathrm{e}^{-2 s \alpha} |a_2 \partial_t \widetilde{y}|^2 \,\d x\,\d t + \int_{\Gamma_{t_0, t_1}} \mathrm{e}^{-2 s \alpha} |\ell_2 \partial_t\widetilde{y}_\Gamma|^2 \,\d S\,\d t \right) \right]. \label{e5}
\end{align}
Since $(y,z)(\theta,\cdot)=(\widetilde{y},\widetilde{z})(\theta,\cdot)$ and by passing to the trace $(y_\Gamma, z_\Gamma)(\theta, \cdot) = (\widetilde{y}_\Gamma, \widetilde{z}_\Gamma)(\theta, \cdot)$, then $u(\theta,\cdot)=a_1 f(\widetilde{y},\widetilde{z})(\theta,\cdot)$ and $u_\Gamma(\theta,\cdot)=\ell_1 g(\widetilde{y}_\Gamma,\widetilde{z}_\Gamma)(\theta,\cdot)$. Hence,
\begin{align}
\lambda^{-4+\epsilon} \left(\int_\Omega \mathrm{e}^{-2 s \alpha(\theta,\cdot)} \xi^{-3}(\theta,\cdot)|a_1 f(\widetilde{y},\widetilde{z})(\theta,\cdot)|^2 \,\d x + \int_\Gamma \mathrm{e}^{-2 s \alpha(\theta,\cdot)} \xi^{-3}(\theta,\cdot) |\ell_1 g(\widetilde{y}_\Gamma,\widetilde{z}_\Gamma)(\theta,\cdot)|^2 \,\d S \right) \nonumber\\
& \hspace{-15.5cm} \leq C \left[ s^{5} \lambda^{2+\epsilon} \int_{\omega_{t_0, t_1}} \mathrm{e}^{-2 s \alpha} \xi^4 |v|^2 \,\d x\,\d t \right. + s^{-2} \lambda^{-6+\epsilon} \left(\int_{\Omega_{t_0, t_1}} \mathrm{e}^{-2 s \alpha} \xi^{-3} |a_1 \partial_t f(\widetilde{y},\widetilde{z})|^2 \,\d x\,\d t \right.\nonumber \\
& \hspace{-14.5cm} + \left. \int_{\Gamma_{t_0, t_1}} \mathrm{e}^{-2 s \alpha} \xi^{-3} |\ell_1 \partial_t g(\widetilde{y}_\Gamma, \widetilde{z}_\Gamma)|^2 \,\d S\,\d t \right) \nonumber\\
& \hspace{-14.5cm} + \left. s\lambda^{-2+2 \epsilon} \left(\int_{\Omega_{t_0, t_1}} \mathrm{e}^{-2 s \alpha} |a_2 \partial_t \widetilde{y}|^2 \,\d x\,\d t + \int_{\Gamma_{t_0, t_1}} \mathrm{e}^{-2 s \alpha} |\ell_2 \partial_t\widetilde{y}_\Gamma|^2 \,\d S\,\d t \right) \right]. \label{e6}
\end{align}
By setting
$$\Phi :=(\phi, \phi_\Gamma) =\mathrm{e}^{-s\alpha} V, \quad \mathbf{I}_2 := \int_{\Omega_{t_0,\theta}} M_1^{(0)} \phi\, \phi \,\d x\,\d t, \quad \mathbf{J}_2 := \int_{\Gamma_{t_0,\theta}} N_1^{(0)} \phi_\Gamma\, \phi_\Gamma \,\d S\,\d t,$$
and using $v(\theta,\cdot)=a_2 \widetilde{y}(\theta,\cdot), v_\Gamma(\theta,\cdot)=\ell_2 \widetilde{y}_\Gamma(\theta,\cdot)$, we deduce that
\begin{align}
\int_\Omega \mathrm{e}^{-2 s \alpha(\theta,\cdot)} |a_2 \widetilde{y}(\theta,\cdot)|^2 \,\d x + \int_\Gamma \mathrm{e}^{-2 s \alpha(\theta,\cdot)} |\ell_2 \widetilde{y}_\Gamma(\theta,\cdot)|^2 \,\d S\nonumber\\
& \hspace{-8.5cm} \leq C \left[ s^{5/2} \lambda^{2+\epsilon} \int_{\omega_{t_0, t_1}} \mathrm{e}^{-2 s \alpha} \xi^4 |v|^2 \,\d x\,\d t \right. \nonumber\\
& \hspace{-8cm} + s^{-9/2} \lambda^{-6+\epsilon} \left(\int_{\Omega_{t_0, t_1}} \mathrm{e}^{-2 s \alpha} \xi^{-3} |a_1 \partial_t f(\widetilde{y},\widetilde{z})|^2 \,\d x\,\d t \right.\nonumber \\
& \hspace{-8cm} + \left. \int_{\Gamma_{t_0, t_1}} \mathrm{e}^{-2 s \alpha} \xi^{-3} |\ell_1 \partial_t g(\widetilde{y}_\Gamma, \widetilde{z}_\Gamma)|^2 \,\d S\,\d t \right) \nonumber\\
& \hspace{-8.5cm} + \left. s^{-3/2}\lambda^{-2+2 \epsilon} \left(\int_{\Omega_{t_0, t_1}} \mathrm{e}^{-2 s \alpha} |a_2 \partial_t \widetilde{y}|^2 \,\d x\,\d t + \int_{\Gamma_{t_0, t_1}} \mathrm{e}^{-2 s \alpha} |\ell_2 \partial_t\widetilde{y}_\Gamma|^2 \,\d S\,\d t \right) \right]. \label{e7}
\end{align}
Adding up \eqref{e6}-\eqref{e7} one can absorb the remaining terms in the right-hand side by choosing $0<\epsilon<1$ and a large parameter $\lambda$. Indeed, since the $C_0$-semigroup $\left(\mathrm{e}^{t\mathcal{A}}\right)_{t\ge 0}$ is analytic, the maximal regularity result in \cite[Corollary 2.13]{An'90} implies that $\widetilde{Y}:=(\widetilde{y}, \widetilde{y}_\Gamma) \in H^1\left(t_0, t_1; \mathbb{H}^2\right)$. Since $N\le 3$, the Sobolev embeddings $H^2(\Omega) \hookrightarrow L^\infty(\Omega)$ and $H^2(\Gamma) \hookrightarrow L^\infty(\Gamma)$ yield that
$\partial_t \widetilde{y} \in L^2\left(t_0, t_1, L^\infty(\Omega)\right)$ and $\partial_t \widetilde{y}_\Gamma \in L^2\left(t_0, t_1, L^\infty(\Gamma)\right)$. Therefore, the positivity \eqref{pineq} and Assumption II-(iii)-(ii) imply that 
$$
\begin{aligned}
& \text{for a.e } (t, x) \in \Omega_{t_0, t_1}:\quad &&|\partial_t \widetilde{y}(t,x)| \leq \frac{\|\partial_t \widetilde{y}(t,\cdot)\|_{L^\infty(\Omega)}}{r}| \widetilde{y}(\theta, x)|, \\
& \quad &&|\partial_t f(\widetilde{y}, \widetilde{z})(t,x)| \leq \frac{\|\partial_t f(\widetilde{y}, \widetilde{z})(t,\cdot)\|_{L^\infty(\Omega)}}{r_1}| f(\widetilde{y}, \widetilde{z})(\theta, x)|, \\
& \text{for a.e } (t, x) \in \Gamma_{t_0, t_1}:\quad &&|\partial_t \widetilde{y}_\Gamma(t,x)| \leq \frac{\|\partial_t \widetilde{y}_\Gamma(t,\cdot)\|_{L^\infty(\Gamma)}}{r}| \widetilde{y}_\Gamma(\theta, x)|\\
& \quad &&|\partial_t g(\widetilde{y}_\Gamma, \widetilde{z}_\Gamma)(t,x)| \leq \frac{\|\partial_t g(\widetilde{y}_\Gamma, \widetilde{z}_\Gamma)(t,\cdot)\|_{L^\infty(\Gamma)}}{r_1}| g(\widetilde{y}_\Gamma, \widetilde{z}_\Gamma)(\theta, x)|.
\end{aligned}
$$
Since $\alpha(t,\cdot) \ge \alpha(\theta,\cdot)$ for all $t\in (t_0,t_1)$, the above estimates allow us to absorb the last four integrals in \eqref{e6}-\eqref{e7} by the left-hand sides. Therefore, by \eqref{pineq} and Assumption II-(ii), we obtain
\begin{align*}
&\lambda^{-4+\epsilon} \left(\int_\Omega \mathrm{e}^{-2 s \alpha(\theta,\cdot)} \xi^{-3}(\theta,\cdot) |a_1|^2 \,\d x + \int_\Gamma \mathrm{e}^{-2 s \alpha(\theta,\cdot)} \xi^{-3}(\theta,\cdot) |\ell_1|^2 \,\d S \right)\\
& \hspace{1cm} +\int_\Omega \mathrm{e}^{-2 s \alpha(\theta,\cdot)} |a_2|^2 \,\d x + \int_\Gamma \mathrm{e}^{-2 s \alpha(\theta,\cdot)} |\ell_2|^2 \,\d S\nonumber\\
& \leq C s^{5} \lambda^{2+\epsilon} \int_{\omega_{t_0, t_1}} \mathrm{e}^{-2 s \alpha} \xi^4 |v|^2 \,\d x\,\d t.
\end{align*}
The functions $\mathrm{e}^{-2 s \alpha(\theta,\cdot)}$ and $\xi^{-3}(\theta,\cdot)$ are bounded from below by positive constants on $\overline{\Omega}$ and $e^{-2 s \alpha} \xi^{4} \text { is bounded on } \Omega_{t_0,t_1}$. Moreover, we can fix $s$ and $\lambda$ sufficiently large. Thus, the proof is completed.
\end{proof}

\begin{remark}
We close the paper with the following remarks:
\begin{itemize}
\item The method we adopted above does not apply for systems like system \eqref{peq1to6} that contain a semilinearity in each equation.
\item For similar inverse problems with boundary measurements on a part $\gamma \Subset \Gamma$ instead of $\omega \Subset \Omega$, the boundary Carleman estimate with one observation for coupled parabolic systems is an open problem even for Dirichlet boundary conditions.
\item It would be of much interest to prove similar stability results for strongly coupled (i.e., coupled through the principal parts) parabolic systems with dynamic boundary conditions in the spirit of \cite{WY'17}.
\end{itemize}
\end{remark}

\appendix
\section{Proof of Lemma \ref{lm2}} \label{app}
Throughout the proof, we combine some ideas from the proofs of \cite[Lemma 2.4]{ACMO'20} and \cite[Theorem 7.1]{Fu'00}, so we shorten similar calculations.

First, let us recall some properties of the weights $\alpha$ and $\xi$ from Section \ref{sec3}:
\begin{enumerate}[label={(\alph*)}]
\item \label{a} $|\partial_t \alpha| \leq CT\xi^2$ and $|\partial_t \xi| \le C T \xi^2$.
\item \label{b} $\xi \geq \dfrac{4}{(t_1 -t_0)^2}$ and $\xi \leq \dfrac{(t_1 -t_0)^4}{16} \xi^3$.
\item \label{c} $\left| \left(\dfrac{\tau}{2}-s\alpha\right) \partial_t \log \gamma(t)\right| \le C_{T,\tau} s \xi^2$ for every $\tau \in \mathbb{R}$.
\item \label{d} $\left|\partial_t \left(\left(\dfrac{\tau}{2}-s\alpha\right) \partial_t \log \gamma(t)\right)\right| \le C_{T,\tau} s \xi^3$ for every $\tau \in \mathbb{R}$.
\end{enumerate}
Furthermore, we set
\begin{align*}
f &:=e^{-s\alpha} L z, \qquad g:= e^{-s\alpha} L_{\Gamma}(z_\Gamma, z),\notag\\
\tilde{f} &:=f-\lambda \left(s\xi +\frac{\tau}{2}\right) \dv\left(A(x)\nabla \eta^0\right) \psi+\left[\lambda^2 \frac{\tau^2}{4} -s \lambda^2 \xi (1-\tau)\right] \sigma \psi,
\end{align*}
so that
\begin{align}
M_1^{(\tau)}\psi + M_2^{(\tau)} \psi=\tilde{f}, \label{e7-28}\\
N_1^{(\tau)}\psi + N_2^{(\tau)} \psi=g \label{e'7-28}.  
\end{align}
Recall that there exists a positive constant $C_1>0$ such that
\begin{equation}\label{sigmaineq}
\beta |\nabla \eta^0|^2 \leq \sigma(x) \leq  C_1, \quad x\in \overline{\Omega},
\end{equation}
and that we have
\begin{align}
\nabla\alpha &= -\nabla \xi=-\lambda \xi \nabla \eta^0, \label{2eq3.3}\\
\dv(A(x)\nabla \alpha) &= -\lambda^2 \xi \sigma -\lambda \xi \dv(A(x)\nabla\eta^0), \nonumber\\
\partial_t \psi &= \mathrm{e}^{-s\alpha} \partial_t z-s\psi \partial_t \alpha, \nonumber\\
\nabla \psi &= \mathrm{e}^{-s\alpha} \nabla z +s\lambda \psi \xi \nabla \eta^0 , \nonumber\\
\dv(A(x)\nabla \psi) &= \mathrm{e}^{-s\alpha} \dv(A(x)\nabla z) + 2s \lambda \xi  A(x)\nabla \eta^0\cdot \nabla\psi - s^2\lambda^2 \xi^2 \psi \sigma \nonumber\\
& \qquad + s\lambda \xi \psi \dv(A(x)\nabla\eta^0)+ s\lambda^2 \xi \psi \sigma .\nonumber
\end{align}
We will often use the following estimates on $\overline{\Omega}$,
\begin{equation}\label{2eq3.9}
|\nabla \alpha| \leq C \lambda \xi, \qquad |\partial_t \alpha| \leq C \xi^2, \qquad |\partial_t \xi| \leq C \xi^2.
\end{equation}
By integration by parts and using \eqref{2eq3.3}, we obtain
\begin{align*}
& \langle M_{1,1}\psi, M_{2,1}\psi\rangle_{L^2(\Omega_T)} =-s^2\lambda^3 \int_{\Omega_T} \left(s\xi +\frac{\tau}{2}\right)\xi^2 \sigma  A(x)\nabla \eta^0  \cdot \nabla (\psi^2) \,\d x\,\d t.\\
& = \int_{\Omega_T} (3 s^3 \lambda^4 \xi^3 + \tau s^2\lambda^4 \xi^2) \sigma^2 \psi^2 \,\d x\,\d t + s^2 \lambda^3 \int_{\Omega_T}  \left(s\xi +\frac{\tau}{2}\right) \xi^2 (\nabla \sigma \cdot A(x) \nabla \eta^0) \psi^2 \,\d x\,\d t\\
& + s^2 \lambda^3 \int_{\Omega_T}  \left(s\xi +\frac{\tau}{2}\right)\xi^2  \sigma  \dv(A(x)\nabla \eta^0) \psi^2 \,\d x\,\d t -s^2\lambda^3 \int_{\Gamma_T} \left(s\xi +\frac{\tau}{2}\right)\xi^2 \sigma \partial_\nu^A \eta^0 \psi^2 \,\d S\,\d t.
\end{align*}
After integrating by parts in time and using \eqref{sigmaineq} with \eqref{2eq3.9}, we obtain
\begin{align*}
\langle M_{1,2}\psi, M_{2,1}\psi\rangle_{L^2(\Omega_T)} &=-\frac{1}{2} s^2 \lambda^2 \int_{\Omega_T} \sigma \xi^2 \partial_t (\psi^2) \,\d x\,\d t = s^2 \lambda^2 \int_{\Omega_T} \sigma \partial_t \xi \xi \psi^2 \,\d x\,\d t
\end{align*}
since $\psi$ vanishes at $t=0$ and $t=T$.

Using integration by parts and $\partial_i (a_{ij} \partial_j \psi)=a_{ij} \partial_i \partial_j \psi + \partial_i (a_{ij}) \partial_j \psi$, with help of \eqref{2eq3.3} the next addend becomes
\begin{align*}
& \langle M_{1,1}\psi, M_{2,2}\psi\rangle_{L^2(\Omega_T)} =-2\lambda \int_{\Omega_T} \left(s\xi +\frac{\tau}{2}\right)(\nabla \eta^0 \cdot A(x)\nabla \psi) \dv(A(x)\nabla \psi) \,\d x\,\d t\\
&= 2s\lambda^2 \bigintsss_{\Omega_T} \xi \left|\nabla \eta^0 \cdot A(x)\nabla \psi\right|^2 \,\d x\,\d t\\
& \qquad +  2\lambda \int_{\Omega_T} \sum_{i,j=1}^N \sum_{k,l=1}^N \left(s\xi +\frac{\tau}{2}\right) \partial_i (a_{ij} a_{kl} \partial_k \eta^0) (\partial_l \psi) (\partial_j \psi) \,\d x\,\d t\\
& \qquad +  2\lambda \int_{\Omega_T} \sum_{i,j=1}^N \sum_{k,l=1}^N \left(s\xi +\frac{\tau}{2}\right)  a_{ij} a_{kl} \partial_k \eta^0 (\partial_i \partial_l \psi) (\partial_j \psi) \,\d x\,\d t \\
& \qquad - 2\lambda \int_{\Gamma_T} \left(s\xi +\frac{\tau}{2}\right) (\partial_\nu \eta^0) (\partial_\nu^A \psi)^2 \,\d S\,\d t\\
& \qquad -2\lambda \int_{\Omega_T} \left(s\xi +\frac{\tau}{2}\right) \sum_{i,j=1}^N \sum_{k,l=1}^N a_{kl} (\partial_k \eta^0) (\partial_l \psi) \partial_i (a_{ij}) \partial_j \psi \,\d x\,\d t\\
& = \mathbf{D}_1+\mathbf{D}_2+\mathbf{D}_3+\mathbf{D}_4 +\mathbf{D}_5.
\end{align*}
By integration by parts, we obtain
\begin{align*}
\mathbf{D}_3 & = \lambda \int_{\Omega_T} \sum_{i,j=1}^N \sum_{k,l=1}^N \left(s\xi +\frac{\tau}{2}\right)  a_{ij} a_{kl} (\partial_k \eta^0) \partial_l [(\partial_i \psi)(\partial_j \psi)] \,\d x\,\d t\\
&= - s\lambda^2 \int_{\Omega_T} \sigma \xi A(x)\nabla \psi\cdot \nabla \psi \,\d x\,\d t\\
& \quad - \lambda \int_{\Omega_T} \left(s\xi +\frac{\tau}{2}\right) \sum_{i,j=1}^N \sum_{k,l=1}^N  \partial_l (a_{ij} a_{kl}\partial_k \eta^0) (\partial_i \psi)(\partial_j \psi) \,\d x\,\d t\\
& \quad + \lambda \int_{\Gamma_T} \left(s\xi +\frac{\tau}{2}\right) (\partial_\nu \eta^0) (A(x)\nu\cdot \nu)(A(x)\nabla \psi \cdot \nabla \psi) \,\d S\,\d t,
\end{align*}
Integrating by parts, we obtain
\begin{align}\label{eqtg}
\langle M_{1,2}\psi, M_{2,2}\psi\rangle_{L^2(\Omega_T)} & = -\int_{\Gamma_T} \partial_t \psi \partial_\nu^A \psi \,\d S\,\d t.
\end{align}
Using integration by parts, \eqref{2eq3.3} and \eqref{2eq3.9} imply
\begin{align}
&\langle M_{1,1}\psi, M_{2,3}\psi\rangle_{L^2(\Omega_T)} = s\lambda \int_{\Omega_T} \left(s\xi +\frac{\tau}{2}\right) \left(\frac{\tau}{2}-s\alpha\right)(\partial_t \log \gamma) A(x)\nabla\eta^0 \cdot \nabla(\psi^2) \,\d x\,\d t \nonumber\\
& \quad = \lambda \int_{\Gamma_T} \left(s\xi +\frac{\tau}{2}\right) \left(\frac{\tau}{2}-s\alpha\right)(\partial_t \log \gamma) \partial_\nu^A \eta^0 \, \psi^2 \,\d S\,\d t \label{st2}\\
& \qquad - \lambda \int_{\Omega_T} \left(s\xi +\frac{\tau}{2}\right) \left(\frac{\tau}{2}-s\alpha\right)(\partial_t \log \gamma) \dv(A\nabla \eta^0)\psi^2 \,\d x\,\d t \nonumber\\
& \qquad - s\lambda^2 \int_{\Omega_T} (\partial_t \log \gamma) \xi(\tau +s(\xi-\alpha))\sigma \psi^2 \,\d x\,\d t \nonumber
\end{align}
Since $\psi(0)=\psi(T)=0$ and $|\partial_t^2 \alpha|\leq C\xi^3$, integration by parts yields
\begin{align*}
\langle M_{1,2}\psi, M_{2,3}\psi\rangle_{L^2(\Omega_T)} &= -\frac{1}{2}\int_{\Omega_T} \partial_t \left(\left(\frac{\tau}{2}-s\alpha\right) (\partial_t \log \gamma)\right)\, \psi^2 \,\d x\,\d t
\end{align*}
For the boundary terms $N_1^{(\tau)}$ and $N_2^{(\tau)}$, we will use the divergence formula \eqref{sdt}. We have
\begin{align*}
\langle N_{1,1}\psi, N_{2,1}\psi\rangle_{L^2(\Gamma_T)}&=0,
\end{align*}
by means of $\psi(0)=\psi(T)=0$. Since $\xi(t,\cdot)$ is constant on $\Gamma$, \eqref{sdt} and \eqref{uellipD} yield
\begin{align*}
&\langle N_{1,2}\psi, N_{2,1}\psi\rangle_{L^2(\Gamma_T)} = \lambda \int_{\Gamma_T} \left(s\xi +\frac{\tau}{2}\right)\partial_\nu^A \eta^0 \psi \dv_\Gamma (D(x)\nabla_\Gamma \psi) \,\d S\,\d t\\
&= -  \lambda \int_{\Gamma_T} \left(s\xi +\frac{\tau}{2}\right) \psi \langle D(x)\nabla_\Gamma \psi, \nabla_\Gamma(\partial_\nu^A \eta^0)\rangle_\Gamma \d S \d t\\
& \quad -  \lambda \int_{\Gamma_T} \partial_\nu^A \eta^0 \left(s\xi +\frac{\tau}{2}\right) \langle D(x)\nabla_\Gamma \psi, \nabla_\Gamma \psi \rangle_\Gamma \d S \d t.
\end{align*}
The next terms are given by by
\begin{align*}
\langle N_{1,1}\psi, N_{2,2}\psi\rangle_{L^2(\Gamma_T)} &=-\frac{1}{2} \int_{\Gamma_T} \partial_t\left( \left(\frac{\tau}{2}-s\alpha\right) (\partial_t \log \gamma)\right)  \psi^2 \,\d S\,\d t
\end{align*}
and by \eqref{2eq3.9} we have
\begin{align*}
\langle N_{1,2}\psi, N_{2,2}\psi\rangle_{L^2(\Gamma_T)} &=- \lambda \int_{\Gamma_T} \left(s\xi +\frac{\tau}{2}\right) \left(\frac{\tau}{2}-s\alpha\right) \partial_\nu^A \eta^0 (\partial_t \log \gamma) \psi^2 \,\d S\,\d t
\end{align*}
This last term cancels with \eqref{st2}, and the term
$$\langle N_{1,1}\psi, N_{2,3}\psi\rangle_{L^2(\Gamma_T)}= \int_{\Gamma_T} \partial_t \psi \partial_\nu^A \psi \,\d S\,\d t$$
simplifies with \eqref{eqtg}. Finally,
$$\langle N_{1,2}\psi, N_{2,3}\psi\rangle_{L^2(\Gamma_T)}=- \lambda\int_{\Gamma_T} \left(s\xi +\frac{\tau}{2}\right) \partial_\nu^A \eta^0 \partial_\nu^A \psi \psi \,\d S\,\d t.$$

Next, we estimate the terms of $\left\langle M_1^{(\tau)} \psi,  M_2^{(\tau)} \psi \right\rangle_{L^2(\Omega_T)}$ that are similar. Using \ref{a} and \ref{d}, we have
\begin{align*}
I_1 &:=\langle M_{1,2}, M_{2,1}\rangle_{L^2(\Omega_T)}+\langle M_{1,2}, M_{2,3}\rangle_{L^2(\Omega_T)} \notag\\
& = s^2 \lambda^2 \int_{\Omega_T} \sigma \partial_t \xi \xi \psi^2 \,\d x\,\d t -\frac{1}{2}\int_{\Omega_T} \partial_t \left(\left(\frac{\tau}{2}-s\alpha\right) (\partial_t \log \gamma)\right)\, \psi^2 \,\d x\,\d t \notag\\
&\ge -C s^2 \lambda^2 \int_{\Omega_T} \xi^3 \psi^2 \,\d x\,\d t.
\end{align*}
Using \ref{c}, we see that
\begin{align*}
& I_2 :=\int_{\Omega_T} (3 s^3 \lambda^4 \xi^3 + \tau s^2\lambda^4 \xi^2) \sigma^2 \psi^2 \,\d x\,\d t + s^2 \lambda^3 \int_{\Omega_T}  \left(s\xi +\frac{\tau}{2}\right) \xi^2 (\nabla \sigma \cdot A(x) \nabla \eta^0) \psi^2 \,\d x\,\d t \notag\\
& \quad + s^2 \lambda^3 \int_{\Omega_T}  \left(s\xi +\frac{\tau}{2}\right)\xi^2  \sigma  \dv(A(x)\nabla \eta^0) \psi^2 \,\d x\,\d t \notag \\
& - \int_{\Omega_T} (\partial_t \log \gamma) \left[\lambda \left(s\xi +\frac{\tau}{2}\right) \left(\frac{\tau}{2}-s\alpha\right) \dv(A\nabla \eta^0) + s\lambda^2 \xi(\tau +s(\xi-\alpha))\sigma \right]\psi^2 \,\d x\,\d t \notag \\
& \quad \ge 3 s^3\lambda^4 \int_{\Omega_T} \sigma^2 \xi^3 \psi^2 \,\d x\,\d t - C s^3\lambda^3 \int_{\Omega_T} \xi^3 \psi^2 \,\d x\,\d t.
\end{align*}
Using \eqref{ceqlm1} of Lemma \ref{lm0}, we can estimate
\begin{align*}
J_2 &:=-s^3\lambda^3\int_{\Gamma_T} \xi^3 \sigma (\partial_\nu^A \eta^0) \psi^2 \,\d S\,\d t + \lambda \int_{\Gamma_T} \left(s\xi +\frac{\tau}{2}\right) \left(\frac{\tau}{2}-s\alpha\right) (\partial_t \log \gamma) (\partial_\nu^A \eta^0) \psi^2 \,\d S\,\d t \notag\\
& \ge Cs^3\lambda^3 \int_{\Gamma_T} \xi^3 \psi^2 \,\d S\,\d t.
\end{align*}
By virtue of \eqref{eqconormal} of Lemma \ref{lm0}, we find
\begin{align*}
I_3 &:= 2s\lambda^2 \bigintsss_{\Omega_T} \xi \left|\nabla \eta^0 \cdot A(x)\nabla \psi\right|^2 \,\d x\,\d t - s\lambda^2 \int_{\Omega_T} \sigma \xi A(x)\nabla \psi\cdot \nabla \psi \,\d x\,\d t \notag\\
& \quad +  2\lambda \int_{\Omega_T} \sum_{i,j=1}^N \sum_{k,l=1}^N \left(s\xi +\frac{\tau}{2}\right) \partial_i (a_{ij} a_{kl} \partial_k \eta^0) (\partial_l \psi) (\partial_j \psi) \,\d x\,\d t \notag \\
& \quad - \lambda \int_{\Omega_T} \left(s\xi +\frac{\tau}{2}\right) \sum_{i,j=1}^N \sum_{k,l=1}^N  \partial_l (a_{ij} a_{kl}\partial_k \eta^0) (\partial_i \psi)(\partial_j \psi) \,\d x\,\d t \notag\\
&\quad -2\lambda \int_{\Omega_T} \left(s\xi +\frac{\tau}{2}\right) \sum_{i,j=1}^N \sum_{k,l=1}^N a_{kl} (\partial_k \eta^0) (\partial_l \psi) \partial_i (a_{ij}) \partial_j \psi \,\d x\,\d t \notag\\
& \quad + \lambda \int_{\Gamma_T} \left(s\xi +\frac{\tau}{2}\right) (\partial_\nu \eta^0) \left[\left|A^{\frac{1}{2}}\nu \right|^2 \left|A^{\frac{1}{2}}\nabla_\Gamma \psi \right|^2 -(A\nabla_\Gamma \psi \cdot \nu)^2\right] \,\d S\,\d t \notag\\
& \quad -\lambda \int_{\Gamma_T} \left(s\xi +\frac{\tau}{2}\right) (\partial_\nu \eta^0) (\partial_\nu^A \psi)^2 \,\d S\,\d t \notag\\
& \ge -C s\lambda^2 \int_{\Omega_T} \xi \sigma |\nabla \psi|^2  \,\d x\,\d t - C s\lambda \int_{\Omega_T} \xi |\nabla \psi|^2  \,\d x\,\d t \nonumber\\
& + C s\lambda \int_{\Gamma_T} \xi (\partial_\nu \eta^0) |\nabla_\Gamma \psi|^2 \,\d S\,\d t -\lambda \int_{\Gamma_T} \left(s\xi +\frac{\tau}{2}\right) (\partial_\nu \eta^0) (\partial_\nu^A \psi)^2 \,\d S\,\d t.
\end{align*}
Now we estimate the boundary terms in $\left\langle N_1^{(\tau)} \psi,  N_2^{(\tau)} \psi \right\rangle_{L^2(\Gamma_T)}$.

Using Cauchy-Schwarz inequality, the estimate \eqref{ceqlm1} and the uniform ellipticity of $D$, we get
\begin{align*}
\langle N_{1,2}\psi, N_{2,1}\psi\rangle_{L^2(\Gamma_T)} &\ge - C  \int_{\Gamma_T} |\xi^{1/2}\nabla_\Gamma \psi| |s\lambda \xi^{1/2}\psi| \, \d S\, \d t -  Cs\lambda \int_{\Gamma_T} \partial_\nu \eta^0 \xi |\nabla_\Gamma \psi|^2 \,\d S\,\d t \\
& \ge -C \int_{\Gamma_T} \xi |\nabla_\Gamma \psi|^2 \,\d S\,\d t -C s^2\lambda^2 \int_{\Gamma_T} \xi^2 \psi^2 \,\d S\,\d t \\
& \quad -  Cs\lambda \int_{\Gamma_T} \partial_\nu \eta^0 \xi |\nabla_\Gamma \psi|^2 \,\d S\,\d t,
\end{align*}
where we employed Young's inequality and the estimate $\xi \le C \xi^2$. Making use of \ref{d} we have
$$\langle N_{1,1}\psi, N_{2,2}\psi\rangle_{L^2(\Gamma_T)} \ge -C s \int_{\Gamma_T} \xi^3 \psi^2 \,\d S\,\d t.$$
Finally, since $\eta^0 \in C^2(\Omega)$, Young's inequality and $\xi \le C \xi^2$ imply that
\begin{align*}
\langle N_{1,2}\psi, N_{2,3}\psi\rangle_{L^2(\Gamma_T)} &\ge  -C  \int_{\Gamma_T} |s \lambda \xi^{1/2}\psi| |\xi^{1/2}\partial_\nu^A \psi| \,\d S\,\d t\\
& \ge -C s^2 \lambda^2 \int_{\Gamma_T} \xi^2 \psi^2 \,\d S\,\d t - C \int_{\Gamma_T} \xi (\partial_\nu^A \psi)^2 \,\d S\,\d t.
\end{align*}

On the other hand, the estimate \ref{d} yields that
\begin{equation*}
\|\tilde{f}\|_{L^2(\Omega_T)}^2 \le \int_{\Omega_T} \left(2f^2 + Cs^2\lambda^4 \xi^2 \psi^2 \right) \,\d x\,\d t.
\end{equation*}
Absorbing lower order terms by choosing $s$ and $\lambda$ large as in \cite{ACMO'20}, we infer that
\begin{align}
&\lVert M_1^{(\tau)}\psi \rVert^2_{L^2(\Omega_T)}+  \lVert M_2^{(\tau)}\psi \rVert^2_{L^2(\Omega_T)} + \lVert N_1^{(\tau)}\psi \rVert^2_{L^2(\Gamma_T)}+  \lVert N_2^{(\tau)}\psi \rVert^2_{L^2(\Gamma_T)} \notag\\
& + 2 \int_{\Omega_T} \left(s^3\lambda^4\xi^3\sigma^2\psi^2 -s\lambda^2\xi \sigma |\nabla \psi|^2 \right) \,\d x\,\d t \notag\\
& + s^3\lambda^3\int_{\Gamma_T} \xi^3\psi^2 \,\d S\,\d t + s\lambda\int_{\Gamma_T} \xi (|\nabla_{\Gamma}\psi|^2 + (\partial_\nu^A \psi)^2) \,\d S\,\d t\nonumber\\
& \leq C \int_{\Omega_T} \left(2 f^2 + Cs^3\lambda^3\xi^3 \psi^2+ Cs\lambda \xi |\nabla \psi|^2\right) \,\d x\,\d t + C \int_{\Gamma_T} g^2 \,\d S\,\d t.\label{e7.28}
\end{align}
Multiplying \eqref{e7-28} and by integrating by parts, we obtain
\begin{align*}
\int_{\Omega_T}& \tilde{f} s\lambda^2 \xi \sigma \psi \,\d x\,\d t = s\lambda^2 \int_{\Omega_T} \xi \sigma  (M_1^{(\tau)} \psi)\psi \,\d x\,\d t - s^3\lambda^4 \int_{\Omega_T} \xi^3 \sigma^2 \psi^2 \,\d x\,\d t \notag\\
& \quad -s\lambda^2 \int_{\Omega_T} \xi \sigma \dv(A\nabla \psi) \psi \,\d x\,\d t \notag\\
& \quad + s\lambda^2 \int_{\Omega_T} \left(\frac{\tau}{2}-s\alpha\right) (\partial_t \log \gamma) \xi \sigma \psi^2 \,\d x\,\d t \notag\\
& \hspace{-0.7cm} = s\lambda^2 \int_{\Omega_T} \xi \sigma  (M_1^{(\tau)} \psi)\psi \,\d x\,\d t - s^3\lambda^4 \int_{\Omega_T} \xi^3 \sigma^2 \psi^2 \,\d x\,\d t + s\lambda^2 \int_{\Omega_T} \xi \sigma A\nabla\psi \cdot \nabla\psi \,\d x\,\d t \notag\\
& - \frac{1}{2} s\lambda^2 \int_{\Omega_T} \dv(A\nabla(\xi\sigma))\psi^2 \,\d x\,\d t + s\lambda^2 \int_{\Omega_T} \left(\frac{\tau}{2}-s\alpha\right) (\partial_t \log \gamma) \xi \sigma \psi^2 \,\d x\,\d t \notag\\
& \quad - s\lambda^2 \int_{\Gamma_T} \xi \sigma \psi \partial_\nu^A \psi\,\d S\,\d t + \frac{1}{2} s\lambda^2 \int_{\Gamma_T} (A\nabla(\xi \sigma)\cdot \nu) \psi^2\,\d S\,\d t .
\end{align*}
Then we have
\begin{align}
& s\lambda^2 \int_{\Omega_T} \xi \sigma A\nabla\psi \cdot \nabla\psi \,\d x\,\d t - s^3\lambda^4 \int_{\Omega_T} \xi^3 \sigma^2 \psi^2 \,\d x\,\d t \notag\\
& = \int_{\Omega_T} \tilde{f} s\lambda^2 \xi \sigma \psi \,\d x\,\d t - s\lambda^2 \int_{\Omega_T} \xi \sigma  (M_1^{(\tau)} \psi)\psi \,\d x\,\d t\notag\\
& +\frac{1}{2} s\lambda^2 \int_{\Omega_T} \dv(A\nabla(\xi\sigma))\psi^2 \,\d x\,\d t - s\lambda^2 \int_{\Omega_T} \left(\frac{\tau}{2}-s\alpha\right) (\partial_t \log \gamma) \xi \sigma \psi^2 \,\d x\,\d t\notag\\
& + s\lambda^2 \int_{\Gamma_T} \xi \sigma \psi \partial_\nu^A \psi\,\d S\,\d t - \frac{1}{2} s\lambda^2 \int_{\Gamma_T} (A\nabla(\xi \sigma)\cdot \nu) \psi^2\,\d S\,\d t \notag \\
& \le \int_{\Omega_T} \left(\frac{1}{6} (M_1^{(\tau)} \psi)^2 + f^2 + C s^2\lambda^4 \xi^2 \psi^2\right) \,\d x\,\d t + s\lambda^2 \int_{\Gamma_T} \xi \sigma \psi \partial_\nu^A \psi\,\d S\,\d t \notag\\
& \quad - \frac{1}{2} s\lambda^2 \int_{\Gamma_T} (A\nabla(\xi \sigma)\cdot \nu) \psi^2\,\d S\,\d t . \label{e7.30}
\end{align}
Multiplying both sides of \eqref{e7.30} by $2$, adding up the resulting inequality to \eqref{e7.28} and absorbing the lower boundary terms, we obtain
\begin{align}
&\lVert M_1^{(\tau)}\psi \rVert^2_{L^2(\Omega_T)}+  \lVert M_2^{(\tau)}\psi \rVert^2_{L^2(\Omega_T)} + \lVert N_1^{(\tau)}\psi \rVert^2_{L^2(\Gamma_T)}+ \lVert N_2^{(\tau)}\psi \rVert^2_{L^2(\Gamma_T)} \nonumber\\
& \quad + 6 s^3\lambda^4\int_{\Omega_T} \xi^3 \sigma^2 \psi^2 \,\d x\,\d t + s\lambda^2\int_{\Omega_T} \xi|\nabla\psi|^2 \,\d x\,\d t + s^3\lambda^3\int_{\Gamma_T} \xi^3\psi^2 \,\d S\,\d t\nonumber\\
& \quad  + s\lambda\int_{\Gamma_T} \xi (|\nabla_{\Gamma}\psi|^2 + (\partial_\nu^A \psi)^2) \,\d S\,\d t\nonumber\\
& \leq C_1 \int_{\Omega_T} \left(f^2 + s^3\lambda^3\xi^3 \psi^2 + s\lambda \xi |\nabla \psi|^2\right) \,\d x\,\d t +  C_1\int_{\Gamma_T} g^2 \,\d S\,\d t. \label{e7-31}
\end{align}
In order to eliminate the gradient term on right-hand side of \eqref{e7-31}, we take $\langle \cdot, \cdot\rangle_{L^2(\Omega_T)}$ of \eqref{e7-28} with $s\lambda \xi \psi$, and similarly to previous calculations we infer that
\begin{align*}
s\lambda \int_{\Omega_T} \xi |\nabla \psi|^2 \,\d x\,\d t \le \frac{1}{2C} \|M_1^{(\tau)}\psi\|^2_{L^2(\Omega_T)} + 3 \int_{\Omega_T} f^2 \,\d x\,\d t + C_1 s^3\lambda^3 \int_{\Omega_T} \xi^3\psi^2 \,\d x\,\d t,
\end{align*}
where $C_1>0$ is the same constant in \eqref{e7-31}. We then arrive at
\begin{align*}
&\lVert M_1^{(\tau)}\psi \rVert^2_{L^2(\Omega_T)}+  \lVert M_2^{(\tau)}\psi \rVert^2_{L^2(\Omega_T)} + \lVert N_1^{(\tau)}\psi \rVert^2_{L^2(\Gamma_T)}+ \lVert N_2^{(\tau)}\psi \rVert^2_{L^2(\Gamma_T)} \nonumber\\
& \quad + s^3\lambda^4\int_{\Omega_T} \xi^3 \sigma^2 \psi^2 \,\d x\,\d t + s\lambda^2\int_{\Omega_T} \xi|\nabla\psi|^2 \,\d x\,\d t + s^3\lambda^3\int_{\Gamma_T} \xi^3\psi^2 \,\d S\,\d t\nonumber\\
& \quad  + s\lambda\int_{\Gamma_T} \xi (|\nabla_{\Gamma}\psi|^2 + (\partial_\nu^A \psi)^2) \,\d S\,\d t\nonumber\\
& \leq C \int_{\Omega_T} \left(f^2 + s^3\lambda^3\xi^3 \psi^2\right) \,\d x\,\d t +  C\int_{\Gamma_T} g^2 \,\d S\,\d t. 
\end{align*}
Since $\sigma \ge C >0$ on $\overline{\Omega\setminus \omega'}$, we have
$$s^3\lambda^3 \int_{\Omega_T} \xi^3 \psi^2 \,\d x\,\d t \le C_1 s^3\lambda^3 \int_{\Omega_T} \xi^3 \sigma^2 \psi^2 \,\d x\,\d t + s^3\lambda^3 \int_{\omega'_T} \xi^3 \psi^2 \,\d x\,\d t,$$
Thus
\begin{align*}
&\lVert M_1^{(\tau)}\psi \rVert^2_{L^2(\Omega_T)}+  \lVert M_2^{(\tau)}\psi \rVert^2_{L^2(\Omega_T)} + \lVert N_1^{(\tau)}\psi \rVert^2_{L^2(\Gamma_T)}+ \lVert N_2^{(\tau)}\psi \rVert^2_{L^2(\Gamma_T)} \nonumber\\
& \quad + s^3\lambda^4\int_{\Omega_T} \xi^3 \psi^2 \,\d x\,\d t + s\lambda^2\int_{\Omega_T} \xi|\nabla\psi|^2 \,\d x\,\d t + s^3\lambda^3\int_{\Gamma_T} \xi^3\psi^2 \,\d S\,\d t\nonumber\\
& \quad  + s\lambda\int_{\Gamma_T} \xi (|\nabla_{\Gamma}\psi|^2 + (\partial_\nu^A \psi)^2) \,\d S\,\d t\nonumber\\
& \leq C s^3\lambda^4 \int_{\omega'_T} \xi^3 \psi^2 \,\d x\,\d t +  C \int_{\Omega_T} f^2 \,\d x\,\d t +  C\int_{\Gamma_T} g^2 \,\d S\,\d t. 
\end{align*}
Multiplying \eqref{e7-28} by $(s\xi)^{-\frac{1}{2}}$, we find
\begin{align*}
& s^{-1}\int_{\Omega_T} \xi^{-1} |\dv(A\nabla \psi)|^2 \,\d x\,\d t \notag\\
&\le C \int_{\Omega_T} \left(s^{-1}\xi^{-1} (M_1^{(\tau)} \psi)^2 + (s^3\lambda^4 \xi^3 \sigma^2 +s^3\xi^3) \psi^2\right) \,\d x\,\d t \notag\\
& \le C \left(s^3\lambda^4 \int_{\omega'_T} \xi^3 \psi^2 \,\d x\,\d t + \int_{\Omega_T} f^2 \,\d x\,\d t \right).
\end{align*}
Analogously, multiplying $M_1^{(\tau)} \psi$ by $s\lambda\xi^{-\frac{1}{2}}$ yields
\begin{align*}
s^{-1}\int_{\Omega_T} \xi^{-1} |\partial_t \psi|^2 \,\d x\,\d t 
& \le C \left(s^3\lambda^4 \int_{\omega'_T} \xi^3 \psi^2 \,\d x\,\d t + \int_{\Omega_T} f^2 \,\d x\,\d t \right).
\end{align*}
The identity $$\dv_{\Gamma}\left(D(x) \nabla_{\Gamma} \psi\right)=-N_{2}^{(\tau)} \psi +\left(\dfrac{\tau}{2}-s \alpha\right)\left(\partial_{t} \log \gamma\right) \psi+\partial_{\nu}^{A} \psi$$ and the estimate \ref{c} imply that
$$
\begin{aligned}
& s^{-1} \int_{\Gamma_{T}} \xi^{-1}\left|\operatorname{div}_{\Gamma}\left(D(x) \nabla_{\Gamma} \psi\right)\right|^{2} \,\d S\,\d t\\
& \leq \frac{1}{2} \|N_{2}^{(\tau)} \psi\|_{L^{2}\left(\Gamma_{T}\right)}^{2}+C s \int_{\Gamma_{T}} \xi^{3} \psi^{2} \,\d S\,\d t +C \int_{\Gamma_{T}} \xi\left(\partial_{\nu}^{A} \psi\right)^{2} \,\d S\,\d t.
\end{aligned}
$$
Moreover, using the identity $\partial_{t} \psi = N_{1}^{(\tau)} \psi + \lambda\left(s \xi+\dfrac{\tau}{2}\right) \partial_{\nu}^{A} \eta^{0} \psi$ and Young's inequality, we deduce that
$$s^{-1} \int_{\Gamma_{T}} \xi^{-1}\left|\partial_{t} \psi\right|^{2} \,\d S\,\d t \leq \frac{1}{2} \|N_{1}^{(\tau)} \psi \|_{L^{2}\left(\Gamma_{T}\right)}^{2} +C s \lambda^{2} \int_{\Gamma_{T}} \xi \psi^{2} \,\d S\,\d t.$$
The rest of the proof follows the same strategy in \cite{ACMO'20}.

\section*{Acknowledgments}
We would like to thank anonymous referees for careful reading and invaluable comments which led to this improved version. We also thank Hichem Ramoul for fruitful discussions.


\begin{thebibliography}{99}

\bibitem{ACM'21}
Ait Ben Hassi~EM, Chorfi~SE, Maniar~L. An inverse problem of radiative potentials and initial temperatures in parabolic equations with dynamic boundary conditions. {\it J Inverse Ill-Posed Probl}. 2022;{\bf 30}:363--378.

\bibitem{ACM'21'}
Ait Ben Hassi~EM, Chorfi~SE, Maniar~L. Identification of source terms in heat equation with dynamic boundary conditions. {\it Math Meth Appl Sci}. 2022;{\bf 45}:2364--2379.

\bibitem{ACMO'20}
Ait Ben Hassi~EM, Chorfi~SE, Maniar~L et al. Lipschitz stability for an inverse source problem in anisotropic parabolic equations with dynamic boundary conditions. {\it Evol Equat and Cont Theo}. 2021;{\bf 10}:837--859.

\bibitem{An'90}
Angenent~SB, Nonlinear analytic semiflows. {\it Proc. Roy. Soc. Edinburgh}. 1990;{\bf 115A}:91--107.

\bibitem{BM'08}
Baudouin~L, Mercado~A. An inverse problem for Schrödinger equations with discontinuous main coefficient. {\it Appl Anal}. 2008;{\bf 87}:1145--1165.

\bibitem{BP'02}
Baudouin~L, Puel~J-P.
Uniqueness and stability in an inverse problem for the Schrödinger equation. {\it Inverse Problems}. 2002;{\bf 18}:1537--1554.

\bibitem{BCGY'09}
Benabdallah~A, Cristofol~M, Gaitan~P et al.
Inverse problem for a parabolic system with two components by measurements of one component. {\it Appl Anal}. 2009;{\bf 88}:683--709.

\bibitem{BGL'07}
Benabdallah~A, Gaitan~P, Le Rousseau~J. Stability of discontinuous diffusion coefficients and initial conditions in an inverse problem for the heat equation. {\it SIAM J. Control Optim}. 2007;{\bf 46}:1849--1881.

\bibitem{BCMO'20}
Boutaayamou~I, Chorfi~SE, Maniar~L et al. The cost of approximate controllability of heat equation with general dynamical boundary conditions. {\it Portugal Math}. 2021;{\bf 78}:65--99.

\bibitem{BK'81}
Bukhgeim~Al, Klibanov~MV. Global Uniqueness of a class of multidimensional inverse problems. {\it Soviet Math Dokl}. 1981;{\bf 24}:244--247.

\bibitem{Ch'03}
Choi~J. Inverse problem for a parabolic equation with space-periodic boundary conditions by a Carleman estimate. {\it J Inverse Ill-Posed Probl}. 2003;{\bf 11}:111--135.

\bibitem{Ch'21}
Chorfi~SE. Inverse problems of some parabolic systems with dynamic boundary conditions. PhD thesis, Cadi Ayyad University, 2021.

\bibitem{CGR'06}
Cristofol~M, Gaitan~P, Ramoul~H. Inverse problems for a $2\times 2$ reaction-diffusion system using a Carleman estimate with one observation. {\it Inverse Problems}. 2006;{\bf 22}:1561--1573.

\bibitem{CGRY'12}
Cristofol~M, Gaitan~P, Ramoul~H et al. Identification of two coefficients with data of one component for a nonlinear parabolic system. {\it Appl Anal}. 2012;{\bf 91}:2073--2081.

\bibitem{EFPT'18}
Egger~H, Fellner~K, Pietschmann~JF et al. Analysis and numerical solution of coupled volume-surface reaction-diffusion systems with application to cell biology. {\it Appl. Math. Comput.}. 2018;{\bf 336}:351--367.

\bibitem{FH'11}
Farkas~JZ, Hinow~P. Physiologically structured populations with diffusion and dynamic boundary conditions. {\it Math Biosci Eng}. 2011;{\bf 8}:503--513.

\bibitem{FGGR'02}
Favini~A, Goldstein~JA, Goldstein~GR et al. The heat equation with generalized Wentzell boundary condition. {\it J Evol Equ}. 2002;{\bf 2}:1--19.

\bibitem{Fu'00}
Fursikov~AV. {\it Optimal Control of Distributed Systems. Theory and Applications}. Translations of mathematical monographs, vol. 187, Amer Math Soc. Providence, RI, 2000.
    
\bibitem{FI'96}
Fursikov~AV, Imanuvilov~OYu. {\it Controllability of Evolution} Equations. Lect Notes Ser. {\bf 34}, Seoul National University, Seoul, 1996.    
    
\bibitem{Go'06}
Goldstein~GR. Derivation and physical interpretation of general boundary conditions. {\it Adv Diff Equ}. 2006;{\bf 11}:457--480.    

\bibitem{Is'17}
Isakov~V. {\it Inverse Problems for Partial Differential Equations}. Springer, Cham, 2017. 

\bibitem{KT'14}
Kaddouri~I, Teniou~DE. Inverse problem for a nonlinear parabolic equation with nonsmooth periodic coefficients. {\it SeMA}. 2014;{\bf 66}:55--69.

\bibitem{KM'19}
Khoutaibi~A, Maniar~L. Null controllability for a heat equation with dynamic boundary conditions and drift terms. {\it Evol Equat and Cont Theo}. 2020;{\bf 9}:535--559.

\bibitem{KMMR'19}
Khoutaibi~A, Maniar~L, Mugnolo~D et al. Parabolic equations with dynamic boundary conditions and drift terms. {\it Mathematische Nachrichten}. 2022;{\bf 295}:1211--1232.

\bibitem{KMO'22}
Khoutaibi~A, Maniar~L, Oukdach~O. Null controllability for semilinear heat equation with dynamic boundary conditions. {\it Discrete Contin. Dyn. Syst. -S}. 2022;{\bf 15}:1525--1546.

\bibitem{KL'13}
Klibanov~MV. Carleman estimates for global uniqueness, stability and numerical methods for coefficient inverse problems. {\it J Inverse Ill-Posed Probl}. 2013;{\bf 21}:477--560.

\bibitem{La'32}
Langer~RE. A problem in diffusion or in the flow of heat for a solid in contact with a fluid. {\it Tohoku Math J}. 1932;{\bf 35}:260--275.

\bibitem{MMS'17}
Maniar~L, Meyries~M, Schnaubelt~R. Null controllability for parabolic equations with dynamic boundary conditions. {\it Evol Equat and Cont Theo}. 2017;{\bf 6}:381--407.
        
\bibitem{MOR'08}
Mercado~A, Osses~A, Rosier~L. Inverse problems for the Schrödinger equation via Carleman inequalities with degenerate weights. {\it Inverse Problems}. 2008;{\bf 24}:015017 (18pp).        

\bibitem{Mie'13}
Mielke~A. Thermomechanical modeling of energy-reaction-diffusion systems, including bulk-interface interactions. {\it Discrete Contin. Dyn. Syst. - S}. 2013;{\bf 6}:479--499.

\bibitem{MS'20}
Morgan~J, Sharma~V. Global existence of solutions to volume-surface reaction diffusion systems with dynamic boundary conditions. {\it Differ. Integral Equ}. 2020;{\bf 33}:113--138.

\bibitem{RR'12}
Rätz~A, Röger~M. Turing instabilities in a mathematical model for signaling networks. {\it Journal of Mathematical Biology}. 2012;{\bf 65}:1215--1244.

\bibitem{SAB'20}
Sakthivel~K, Arivazhagan~A, Barani~Balan~N. Inverse problem for a Cahn-Hilliard type system modeling tumor growth. {\it Appl Anal}. 2022;{\bf 101}:858--890.
        
\bibitem{Sa'20}
Sauer~N. Dynamic boundary conditions and the Carslaw-Jaeger constitutive relation in heat transfer. {\it SN Partial Differ Equ Appl}. 2021;{\bf 2}:14.    
        
\bibitem{MS'16}
Sharma~V, Morgan~J. Global existence of solutions to reaction-diffusion systems with mass transport type boundary conditions. {\it SIAM J Math Anal}. 2016;{\bf 48}:4202--4240.

\bibitem{Sm'83}
Smoller~J. {\it Shock Waves and Reaction-Diffusion Equations}. Springer-Verlag, Berlin, 1983.

\bibitem{WY'17}
Wu~B, Yu~J. Hölder stability of an inverse problem for a strongly coupled reaction-diffusion system. {\it IMA J Appl Math}. 2017;{\bf 2}:424--444.
    
\bibitem{YZ'01}
Yamamoto~M, Zou~J. Simultaneous reconstruction of the initial temperature and heat radiative coefficient. {\it Inverse Problems}. 2001;{\bf 17}:1181--1202.

\end{thebibliography}
\end{document}